\documentclass[12pt]{amsart}
\usepackage{amsmath, amsthm, amscd, amssymb, amsfonts, latexsym}
\usepackage{fullpage}
\usepackage{bm}
\usepackage[all]{xy}
\usepackage{tikz}
\usepackage{fancybox,soul,marvosym}
\usepackage{hyperref}

\usepackage{bbold}

\renewcommand{\1}{\mathbb{1}}

\newcommand{\kommentar}[1]{}

\renewcommand{\d}{\delta}

\newcommand{\F}{\mathbb F}

\newcommand{\T}{\mathrm{Trun}^L}

\renewcommand{\pmod}[1]{\,(\mathrm{mod}\,#1)}

\newtheorem{lem}{Lemma}[section]

\newtheorem{thm}[lem]{Theorem}
\newtheorem{defn}[lem]{Definition}

\newtheorem{conj}[lem]{Conjecture}
\newtheorem{question}[lem]{Question}
\theoremstyle{definition}

\newtheorem{fact}[lem]{Expected Fact}
\newtheorem{rem}[lem]{Remark}

\author{Vivian Kuperberg}
\author{Matilde Lal\'in}
\address{Vivian Kuperberg: ETH Z\"urich - Departement Mathematik,
R\"amistrasse 101,
8092 Z\"urich,
Switzerland} \email{vivian.kuperberg@math.ethz.ch}

\address{Matilde Lal\'in:  D\'epartement de math\'ematiques et de statistique, Universit\'e de Montr\'eal, CP 6128, succ. Centre-ville, Montreal, QC H3C 3J7, Canada}\email{matilde.lalin@umontreal.ca}

\subjclass[2020]{Primary 11N05; Secondary 11T06, 11A41}
\keywords{prime truncations; irreducible polynomial truncations; ditribution of primes}

\title{Distributions of left prime truncations}

\begin{document}
\begin{abstract}
The prime number 357686312646216567629137 is notable because of the unusual property that it remains prime successively on removing the left digit until there are no remaining digits. We explore here the distributions of the number of left prime truncations of integers and of the number of irreducible truncations of polynomials with coefficients over a finite field, focusing on the proportion among all $\ell$-digit numbers or polynomials, their variance, and the maximal proportion.
 \end{abstract}

\maketitle
\section{Introduction}

The number 
\begin{equation}\label{eq:longest-left-truncatable-prime-in-base-10}
357686312646216567629137
\end{equation}
has the unusual property that every one of its \emph{left truncations} in base $10$ is prime; that is, for all $k \le 24$, its last $k$ digits constitute a prime number. A number with this property, and no digits equal to $0$, is known in the recreational math literature \cite{AngellGodwin,oeis-lefttruncatable} as a \emph{left-truncatable} prime, and \eqref{eq:longest-left-truncatable-prime-in-base-10} is the longest example in base $10$. It is a folklore conjecture \cite{oeis-totalnumber} that for any $b \ge 2$, there are only finitely many left-truncatable primes in base $b$. Angell and Godwin \cite{AngellGodwin} heuristically suggest that the maximal length of a left-truncatable prime in base $b$ is
\begin{equation*}
\sim \frac{b^2e}{\phi(b)\log b}.
\end{equation*}

More generally, one can simply consider the set of all (left) truncations of an integer $n$ written in base $b$, which leads to a broad class of questions about the distribution of prime truncations. We begin by defining:
\begin{defn}
Let $b \ge 2$ and let $n \in \mathbb N$ where the base $b$ expansion of $n$ is $n = \sum_{j = 0}^\ell a_jb^j$ for $0 \le a_j \le b-1$ and $a_\ell \ne 0$. A \emph{truncation} of $n$ is any integer of the form $\sum_{j=0}^k a_jb^j$, where $0 \le k \le \ell$ and $a_k \ne 0$.
\end{defn}
For example, when $b = 10$, the number $34913$ has truncations $34913$, $4913$, $913$, $13$, and $3$, whereas $83003$ has truncations $83003$, $3003$, and $3$.

It is natural to ask how many prime truncations a typical integer has, as well as to investigate the distribution of the number of prime truncations of $n$ as $n$ varies. Note that the case of left-truncatable primes concerns precisely those $n < b^\ell$ with $\ell$ prime truncations.
\begin{question}\label{qn:counting-truncations-integers}
What proportion of the truncations of an $\ell$-digit number in base $b$ are prime? Precisely, what is the typical proportion among all $\ell$-digit numbers, the variance, and the maximal proportion?
\end{question}
Questions similar to this were studied, for example, by G\'omez-P\'erez and Shparlinski \cite{GomezPerezShparlinski}, who considered the number of prime truncations of a fixed infinite digit sequence in base $b$ and constructed sequences where the limiting proportion of the number of truncations that are prime is large.  They also provided an upper bound for the number of sequences whose truncations are all prime.

In parallel, one can ask the same questions in $\mathbb F_q[T]$, where $\mathbb F_q$ is a finite field of $q$ elements.
\begin{defn}\label{defn:polynomial-truncations-base-T}
Let $\mathbb F_q$ be a finite field and let $f(T) \in \mathbb F_q[T]$ be given by $f(T) := \sum_{j = 0}^\ell a_j T^j$, where $a_j \in \mathbb F_q$ and $a_\ell \ne 0$. A \emph{$T$-truncation} of $f$ is any polynomial of the form $\sum_{j=0}^i a_j T^j$, where $a_i \ne 0$ and $1 \le i \le \ell$.
\end{defn}
Mullen and Shparlinski \cite[Problem 31]{MullenShparlinski} considered for a (possibly infinite) sequence of elements in $\F_q[T]$, the successive polynomials that can be formed by truncating the sequence. They asked for lower and upper bounds for the maximum length of a consecutive irreducible sequence of polynomials. A lower bound was found by Chou and Cohen \cite[Theorem 1.2]{ChouCohen}. These sequences and their factorization properties were further studied by G\'omez-P\'erez, Ostafe, and Sha \cite{GomezPerezOstafeSha}.

In this setting we consider the direct analog of Question \ref{qn:counting-truncations-integers}:
\begin{question}\label{qn:counting-truncations-polynomials}
What proportion of the $T$-truncations of a polynomial in $\mathbb F_q[T]$ of degree $\ell$ are irreducible? Precisely, what is the typical proportion among all degree-$\ell$ polynomials, the variance, and the maximal proportion?
\end{question}

As is common for questions about polynomials over finite fields, there are two natural limits to consider in order to get an asymptotic answer to Question \ref{qn:counting-truncations-polynomials}: namely, the limit as $\ell \to \infty$ and the limit as $q \to \infty$. One unusual aspect of this problem is that \emph{both} limits immediately appear to have natural analogs in the integer setting! For Question \ref{qn:counting-truncations-integers}, either the number of digits $\ell$ or the base $b$ can be taken to grow to infinity; the former is analogous to the degree of a polynomial $f$ increasing and the latter can be taken to be analogous to the size $q$ of the coefficient ring growing large.

However, the analogy between the integer and polynomial setting can be made more concrete by instead considering a generalization of Question \ref{qn:counting-truncations-polynomials}. We can interpret the expression $f(T) = \sum_{j = 0}^\ell a_j T^j$ as a base-$T$ expansion of the polynomial $f$. For any $b(T) \in \mathbb F_q[T]$, one can instead consider a base-$b(T)$ expansion of $f$, given by $f(T) = \sum_{j=0}^\ell \tilde{a_j}(T) b(T)^j$ for appropriate coefficients $\tilde{a_j}(T)$. Accordingly, we can make the following generalization of Definition \ref{defn:polynomial-truncations-base-T}:
\begin{defn}
Let $\mathbb F_q$ be a finite field, and fix $b(T) \in \mathbb F_q[T]$. Let $f(T) \in \mathbb F_q[T]$ be given by $f(T) := \sum_{j = 0}^\ell a_j(T) b(T)^j$, where $a_j(T) \in \mathbb F_q[T]$, $\deg(a_j)<\deg(b)$, and $a_\ell \ne 0$. A \emph{$b(T)$-truncation} of $f$ is any polynomial of the form $\sum_{j=0}^i a_j(T) b(T)^j$, where $a_i \ne 0$ and $1 \le i \le \ell$.
\end{defn}
We will sometimes refer to $b(T)$-truncations as $b$-truncations or even simply truncations, when it is clear from context. We refer to the coefficients $a_j \in \mathbb F_q[T]$ as the \emph{digits} of $f(T)$ in base $b(T)$.

It is then natural to ask about $b(T)$-truncations of polynomials in $\mathbb F_q[T]$, generalizing Question \ref{qn:counting-truncations-polynomials}. This generalization notably introduces another set of limiting behavior, when $|b(T)| \to \infty$, or equivalently when $m \to \infty$ where $b$ is a polynomial of degree $m$. 

The large-$m$ limit in the polynomial case serves as an immediate parallel to the large-$b$ limit in the integer case. We shall see in general that estimates in both the large-$m$ limit and the large-$q$ limit over $\mathbb F_q[T]$ bear resemblance to the large-$b$ limit in $\mathbb Z$, but exhibit substantial differences from the large-$\ell$ limit in their respective settings. Throughout, as we consider results in various settings with $b, \ell, m$ or $q \to \infty$, the implied constants in our error terms will be independent of all of these parameters unless explicitly specified otherwise. That is to say, although we generally specify for each identity which variable is tending to infinity so that the written main term is larger than the written error term, all equations should hold in any limit unless explicitly specified otherwise.

The goal of this paper is to provide answers, some proven and some conjectural, to Questions \ref{qn:counting-truncations-integers} and \ref{qn:counting-truncations-polynomials}.

Given a positive integer $n$, let $\T_b(n)$ be the number of primes among the left truncations of $n$ in base $b$. For example, $\T_{10}(34913)=3$ because $34913, 13,$ and $3$ are prime, while $4913$ and $913$ are not, while $\T_{10}(83003)=2$ because $83003$ and $3$ are prime, while $3003$ is not.

We consider the average of the number of left truncations over all the numbers with $\ell$ digits in base $b$: namely,
\begin{equation}\label{eq:defn-of-avg-truncations-integer-case}
\langle \T_b\rangle_{b^\ell}:=\frac{1}{b^\ell}\sum_{n< b^\ell} \T_b(n).
\end{equation}
\begin{thm}\label{thm:averagenf}
As $b\rightarrow \infty$, we have
\[\langle \T_b\rangle_{b^\ell}=\frac{1}{\log b}\sum_{h=1}^\ell \frac{1}{h}+\frac{1}{\log^2 b}\sum_{h=1}^\ell \frac{1}{h^2} +O\left(\frac{1}{\log^3 b}\right),\]
and as $\ell\rightarrow \infty$, we have
\[\langle \T_b\rangle_{b^\ell}=\left(1-\frac{1}{b}\right)\frac{\log \ell}{\log b}+O\left(\frac 1{\log b}\right).\]
\end{thm}
The main term in each of these estimates, which both follow from Lemma \ref{lem:meannumber}, comes from an asymptotic evaluation of $\left(1-\frac 1b \right) \sum_{h=1}^\ell \frac{\pi(b^h)}{b^h},$ where $\phi$ is the prime counting function.

Given $b(T) \in \mathbb F_q[T]$ of degree $m \ge 1$, we similarly define $\T_{q,b}(f)$, the number of left prime $b(T)$-truncations for a polynomial $f \in \F_q[T]$, and study its average over all the polynomials with up to $\ell$ digits in base $b(T)$ (that is, polynomials of degree $d$ with $d < m\ell$). This average is defined as
\begin{equation}\label{eq:defn-average-truncations-in-polynomial-case}
\langle \T_{q,b}\rangle_{b^\ell}:=\frac{1}{q^{m\ell}}\sum_{\substack{f\in \F_q[T]\\\deg(f)<m \ell}} \T_{q,b}(f).
\end{equation}
In this direction, we prove the following result.
\begin{thm}\label{thm:averageff}
For $b(T) \in \mathbb F_q[T]$ of degree $m \ge 1$, the average number of left prime $b(T)$-truncations for polynomials with $\ell$ digits in base $b(T)$ is asymptotically given by
\begin{equation*}
\langle \T_{q,b}\rangle_{b^\ell} = \sum_{\substack{1 \le h < m \ell \\ h \equiv -1 \pmod m}} \frac 1h + O\left(\frac{\log \ell}{mq} + \frac 1q \right) \text{ as } q \to \infty,
\end{equation*}
\begin{equation*}
\langle \T_{q,b}\rangle_{b^\ell} = \frac 1m \sum_{h=1}^\ell \frac 1h + \frac{q}{(q-1)m^2} \sum_{h=1}^\ell \frac 1{h^2} + O\left(\frac 1{m^3} + \frac{q \log^2(m\ell)}{q^{\frac{m}{2}}}\right) \text{ as } m \to \infty,
\end{equation*}
and
\begin{equation*}
\langle \T_{q,b}\rangle_{b^\ell} = \frac{1}m \left(1-\frac{1}{q^m}\right)\log \ell + O\left(\frac 1m + \frac 1q \right) \text{ as } \ell \to \infty.
\end{equation*}
\end{thm}

The structure of the asymptotic in the polynomial case as $m \to \infty$ (and to a lesser extent as $q \to \infty$) closely mirrors that of the asymptotic in the integer case as $b \to \infty$, where the factor of $\frac 1{\log b}$ in the integer case becomes the factor of $\frac 1m$ in the polynomial case, and the quantity $b$ in the integer case is replaced by the norm of $b(T)$ (i.e. $q^m$) in the polynomial setting.
The analogy between the polynomial case and the integer case as $\ell\to \infty$ is more direct.
Curiously, the factor of $\frac q{q-1}$ in the second term of the $m \to \infty$ estimate in $\mathbb F_q[T]$ does not have a clear parallel in the integer case.

These results are cleaner in the special case when $b(T) = T$ and $m = 1$, where we have
\begin{equation*}
\langle \T_{q,T} \rangle_{T^\ell} = \frac{q-1}{q} \sum_{1 \le h < \ell} \frac 1h + O\left(\frac{1}{q}\right),
\end{equation*}
where the implied constant is independent of $\ell$.

We also consider the variance of the number of left truncations over all numbers with $\ell$ digits in base $b$, given by
\begin{equation}\label{eq:varnf}
 \mathrm{Var}_{b^\ell}(\T_b):=\frac{1}{b^\ell} \sum_{n< b^\ell} \left(\T_b(n)-\langle \T_b\rangle_{b^\ell}\right)^2=\frac{1}{b^\ell} \sum_{n< b^\ell}\T_b(n)^2 -\langle \T_b\rangle_{b^\ell}^2.
\end{equation}
The variance is difficult to evaluate unconditionally in the integer setting, as it relies on estimates of primes in large arithmetic progressions. Roughly speaking, evaluating the variance involves counting primes less than $x = b^\ell$ that are $\equiv p \mod b^i$, where $1 \le i < \ell$ and $p \le b^i$ ranges over primes. The modulus of the arithmetic progression is then generally a (possibly large) power of $b^\ell$ (that is, $x$), which goes beyond unconditional results on primes in arithmetic progressions (see for example \cite[Theorem 5.13]{IwaniecKowalski}). Some better unconditional estimates (for example \cite{banks-shparlinski-MR4042874}) are known for particularly structured moduli, such as when the modulus is a power of a fixed prime; it would be interesting to investigate applications of these results to our setting but at the moment they seem unable to give unconditional results. Moreover, as $\ell \to \infty$, values of $i$ that are close to $\ell$ correspond to moduli that are nearly as large as $x$ itself. Accordingly the variance is very difficult to estimate, even conditionally, in the large-$\ell$ limit. However, as $b$ grows large, an asymptotic estimate is possible assuming the Generalized Riemann Hypothesis:
\begin{thm}\label{thm:varnf} Assume the Generalized Riemann Hypothesis. As $b \rightarrow \infty$, we have,
 \begin{align*}
 \mathrm{Var}_{b^\ell}(\T_b)=
 & \frac{1}{\log b}\sum_{h=1}^\ell \frac{1}{h}+\frac{1}{\log^2 b}\left(\frac{b}{\phi(b)} -1\right)\left[\left(\sum_{h=1}^\ell \frac{1}{h}\right)^2-\sum_{h=1}^\ell \frac{1}{h^2}\right]\\
 &+O\Bigg(\frac{\log \ell \log \log b}{\log^3 b} + \frac{\ell^5 \log^3 b}{\sqrt b}\Bigg),
 \end{align*}
 where $\phi$ is Euler's totient function, and the implied constant in the error term is independent of $\ell$ and $b$.
\end{thm}

In the large $\ell$ limit, we provide heuristic evidence for a conjectural evaluation of the variance. This conjecture appears difficult to prove, as it seems one would need to strongly leverage the fact that the variance is an average over counts of primes in different arithmetic progressions.
\begin{conj} \label{conj:varnf} As $\ell \rightarrow \infty$, we expect that
  \begin{align*}
 \mathrm{Var}_{b^\ell}(\T_b)
 = & \frac{1}{\log^2 b}\left(1 -\frac{1}{b}\right)^2\left(\frac{b}{\phi(b)}-1\right)\log^2\ell+O_b \left(\log \ell \right).
 \end{align*}
\end{conj}
Notably, we conjecture that the limiting behavior of the variance as $b$ or $\ell$ grows large is dominated by different terms (and, perhaps, different phenomena). The main term in Theorem \ref{thm:varnf} has size about $\log \ell$, and would disappear in the large-$\ell$ limit compared to the conjectured main term of size about $\log^2\ell$. This seems to reflect a clustering phenomenon: when an integer $n$ is not coprime to the base $b$, then none (or almost none) of its truncations are prime, so the prime truncations are more densely packed among those $n$ that are coprime to $b$. This phenomenon disappears as $b \to \infty$, but for fixed $b$, we will see a difference when we restrict our attention to those $n$ which are coprime to $b$. We discuss this in more detail after the heuristic leading to Conjecture \ref{conj:varnf} in Section \ref{sec: the integer setting}.

We also study the variance in the polynomial setting, defined by
\begin{align}\label{eq:variance-definition-gen}
\mathrm{Var}_{b^\ell}(\T_{q,b}):=& \frac 1{q^{m\ell}}
\sum_{\substack{f\in \F_q[T]\\\deg(f)<m \ell}}
\left(\T_{q,b}(f)-\langle \T_{q,b}\rangle_{b^\ell}\right)^2\nonumber\\=&\frac 1{q^{m\ell}}
\sum_{\substack{f\in \F_q[T]\\\deg(f)<m \ell}}\T_{q,b}(f)^2-\langle \T_{q,b}\rangle_{b^\ell}^2.
\end{align}
Analogously to the integer setting, we are able to estimate the variance when $q \to \infty$ or when $m \to \infty$. We also provide a heuristic towards a conjecture for the $\ell \to \infty$ limit.
\begin{thm}\label{thm:varff}
Let $q$ be a prime power, let $b \in \mathbb F_q[T]$ be a polynomial of degree $m \ge 1$, and let $\ell \ge 1$. Define $\mathrm{Var}_{b^\ell}(\T_{q,b})$ as in \eqref{eq:variance-definition-gen}. As $q \to \infty$, $\mathrm{Var}_{b^\ell}(\T_{q,b})$ is given by
\begin{multline}\label{eq: variance polynomials q large}
\sum_{\substack{1\leq h <m\ell \\ h\equiv -1\pmod{m}}}\left(\frac{1}{h}-\frac{1}{h^2}\right)\\
+O\left(\ell q^{-\frac{m}{2}}+\mathbb 1_{m \ge 2}\frac{\log\ell q^{\frac{m+1}{2}}}{m^2\Phi(b)}+\frac{\log^2\ell}{mq}+\frac{1}{q}+\left|\frac{q^m}{\Phi(b)}-1\right|\frac{\log^2\ell}{m^2}\right),
\end{multline}
and as $m \to \infty$, $\mathrm{Var}_{b^\ell}(\T_{q,b})$ is given by
\begin{multline}\label{eq: variance polynomials m large}
\frac{1}{m}\sum_{h=1}^\ell \frac{1}{h}+\frac{1}{m^2}\left(\frac{q^m}{\Phi(b)}-1\right)\left(\sum_{h=1}^\ell \frac{1}{h}\right)^2-\frac{1}{m^2}\left(\frac{q^m}{\Phi(b)}-\frac{q}{q-1}\right) \sum_{h=1}^\ell \frac{1}{h^2}\\
+O\left(\frac{\log \ell \log \log (q^m)}{m^3}+\ell q^{-\frac{m}{2}} + q^{1-\frac m2} \log^2(m\ell) \Big(1 + \frac{\log \ell}{m}\Big)\right).\end{multline}
\end{thm}
The above result holds unconditionally, since the Generalized Riemann Hypothesis is a theorem in this context. In the case when $m = 1$, $b(T) = T$, and $q \to \infty$, the variance is given by the rather simpler 
\begin{equation*}
\mathrm{Var}_{T^\ell}(\T_{q,T})=\sum_{\substack{1\leq h <\ell }}\left(\frac{1}{h}-\frac{1}{h^2}\right)+O\left(\ell q^{-1/2}\right).
\end{equation*}

We will provide evidence for the following conjecture concerning the behavior of $\mathrm{Var}_{b^\ell}(\T_{q,b})$ as $\ell \to \infty$.
\begin{conj} \label{conj:varff}
With the variance $\mathrm{Var}_{b^\ell}(\T_{q,b})$ defined in \eqref{eq:variance-definition-gen}, we expect, as $\ell \rightarrow \infty$,
\begin{align*}
\mathrm{Var}_{b^\ell}(\T_{q,b})
&=\frac{1}{m^2} \left(1-\frac{1}{q^m}\right)^2\left(\frac{q^m}{\Phi(b)}-1\right) \log^2 \ell+O_{m,q}\left(\log \ell\right).
\end{align*}
In particular, when $m=1$ and $b(T) = T$, we expect as $\ell \rightarrow \infty$ that
\begin{align*}
\mathrm{Var}_{T^\ell}(\T_{q,T}) =&\frac{q-1}{q^2}\log^2 \ell +O_{q}\left(\log \ell\right).
\end{align*}
\end{conj}
We can again compare the (sometimes conjectural) variance in the integer and polynomial settings. Just as for the average, the variance in the large-$\ell$ limit for the two settings are very closely related. Similarly the large-$m$ limit for polynomials closely resembles the large-$b$ limit for integers, with the caveat that the polynomial case has an extra term of the form $\frac 1{m^2(q-1)}\sum_{h=1}^\ell \frac 1{h^2}$. The large-$q$ limit also bears resemblance to the large-$m$ limit for polynomials and the large-$b$ limit for integers.

Finally, we explore the \emph{maximal} proportion of prime truncations in both the integer and polynomial settings. For a fixed number of digits $\ell$ in either setting, we expect that for a large enough coefficient set (so as $b \to \infty$ in the integers or as $m$ or $q \to \infty$ in $\mathbb F_q[T]$), we will always be able to find an integer (or polynomial) with $\ell$ prime truncations. On the other hand as $\ell \to \infty$, we expect in the integer case that the maximal number of prime truncations is $\asymp \frac{\ell \log b}{\log \ell}$, whereas in the polynomial case that the maximal number of irreducible truncations is $\asymp \frac{m \ell \log q}{\log \ell}$. All of these predictions come from the Cram\'er random model for prime numbers. They are related to the purely probabilistic problem of estimating the maximum out of $b^\ell$ samples of the sum of $\ell$ Bernoulli random variables, where the $k$th variable is $1$ with probability $\frac 1k$ and $0$ otherwise. These sums of Bernoulli random variables are known in the probability literature as \emph{records} (for example, they relate to the number of years with record high rainfall over a fixed span of time; see \cite{Arnold-Balakrishnan-Nagaraja-MR1628157}). However, we are concerned with the maximum of $b^\ell$ samples, which is exponential in the number of random variables we are summing. This corresponds to studying the extreme tail of the distribution of this sum, which does not seem to have been previously studied in a relevant range. Our conjectures are heuristically supported in Sections \ref{sec:maxnf} and \ref{sec:maxff}.

One could ask similar questions about right prime truncations \cite{vanderPoorten}, which highlight a clear difference between the integer and polynomial cases. Indeed, the problem for polynomials with base $b(T)=T$ is entirely symmetric, as a polynomial is irreducible if and only if its reciprocal, the polynomial resulting from reversing the order of the coefficients, is also irreducible. However, this is not necessarily the case when considering an integer in base $b$ or a polynomial in base $b(T)\not =T$.  In the integer case, right prime truncations have been studied by \cite{Dubickas, MillerREU}. Ideas involving mirroring of primes also appear in \cite{GomezPerezShparlinski}. We intend to explore the analogues of Questions \ref{qn:counting-truncations-integers} and \ref{qn:counting-truncations-polynomials} for right prime truncations in forthcoming work.

\section*{Acknowledgments} The authors are grateful to Maxim Gerspach and Igor Shparlinski for helpful comments. This work is supported by the NSF Mathematical Sciences Research Program through the grant DMS-2202128,  the Natural Sciences and Engineering Research Council of Canada, RGPIN-2022-03651, and the Fonds de recherche du Qu\'ebec - Nature et technologies, Projet de recherche en \'equipe 345672.

\section{The integer setting}\label{sec: the integer setting}

Recall that, given $n$ a positive integer, $\T_b(n)$ denotes the number of (left) prime truncations of $n$. That is,
\[\T_b(n):=\sum_{\substack{d\leq n\\ d\equiv n \pmod{b^{\lfloor \log_b d \rfloor +1 }}}}\1_\mathbb{P}(d),\]
where $\1_\mathbb{P}$ denotes the indicator function for the primes. Any positive integer $d$ has $\lfloor\log_b d\rfloor+1$ digits in base $b$, so $d \equiv n \pmod{b^{\lfloor \log_b d \rfloor +1 }}$ if and only if the last digits of $n$ coincide with $d$.

We start by proving Theorem \ref{thm:averagenf}; that is, by estimating the average number $\langle \T_b \rangle_{b^\ell}$ of left truncations over all integers with $\ell$ digits in base $b$ (see \eqref{eq:defn-of-avg-truncations-integer-case}).
\begin{lem} \label{lem:meannumber} We have
\begin{equation}\label{eq:Tidentity}
\langle \T_b\rangle_{b^\ell}=  \left(1 -\frac 1{b}\right)\sum_{h=1}^{\ell - 1}\frac {\pi(b^h)}{b^h} + \frac{\pi(b^\ell)}{b^\ell}.
\end{equation}
\end{lem}
Lemma \ref{lem:meannumber} immediately implies both statements in Theorem \ref{thm:averagenf} via an application of the prime number theorem.

\begin{proof}
By expanding the definition of $\langle \T_b \rangle_{b^\ell}$, swapping the order of summation, and separating by the number of digits of $d$, we get
 \begin{align*}
\langle \T_b\rangle_{b^\ell}=&\frac{1}{b^\ell}\sum_{d< b^\ell} \1_\mathbb{P}(d)   \sum_{\substack{n< b^\ell\\ n\equiv d \pmod{b^{\lfloor\log_b d\rfloor+1 }}}}1 = \frac{1}{b^\ell} \sum_{h=1}^\ell \sum_{b^{h-1}\leq d< b^h} \1_\mathbb{P}(d)   \sum_{\substack{n< b^\ell\\ n\equiv d \pmod{b^{h}}}}1.
\end{align*}
The inside sum is precisely $b^{\ell-h}$, which implies that
\begin{align*}
\langle \T_b\rangle_{b^\ell} =& \sum_{h=1}^\ell b^{-h} (\pi(b^{h}) - \pi(b^{h-1})),
\end{align*}
which can be rearranged to obtain \eqref{eq:Tidentity}. 
\end{proof}

We next consider the variance, defined in \eqref{eq:varnf}.
In estimating the variance we will make use of the following lemma, proven in Section \ref{sec:proofoflemmanf}.
\begin{lem}\label{lem:11} Assume the Generalized Riemann Hypothesis. For integers $0<h_1<h_2$, and for $b \ge 2$, we have
\begin{align*}
\sum_{\substack{b^{h_1-1}\leq d_1 <b^{h_1}\\b^{h_2-1}\leq d_2 <b^{h_2}\\d_2\equiv d_1\pmod{b^{h_1}}}}  \1_{\mathbb{P}}(d_1)\1_{\mathbb{P}}(d_2)=&\frac{b^{h_2+1}}{\phi(b)h_1h_2\log^2 b}\left(1 + O\left(\frac 1{h_1\log b} + \frac 1b + \frac{b^{\frac{h_1-h_2}2}\phi(b)}{b}h_1^3h_2^2\log^5b \right)\right),
 \end{align*}
where the implied constants in the error terms are independent of $b$, $h_1$, and $h_2$.
\end{lem}

With this lemma in hand, we are prepared to prove Theorem \ref{thm:varnf}.
\begin{proof}[Proof of Theorem \ref{thm:varnf}]
By expanding the square we get that
\begin{align*}
\frac{1}{b^\ell} \sum_{n< b^\ell} \T_b(n)^2 = & \frac{1}{b^\ell} \sum_{n< b^\ell} \sum_{\substack{d\leq n\\ d\equiv n \pmod{b^{\lfloor \log_b d \rfloor +1 }}}} \1_{\mathbb{P}}(d)+ \frac{2}{b^\ell} \sum_{n< b^\ell} \sum_{\substack{1\leq d_1<d_2\leq n\\ d_1\equiv n \pmod{b^{\lfloor \log_b d_1 \rfloor +1 }}\\ d_2\equiv n \pmod{b^{\lfloor \log_b d_2 \rfloor +1 }} }}
 \1_{\mathbb{P}}(d_1)\1_{\mathbb{P}}(d_2)\\
 = &\langle \T_b\rangle_{b^\ell} +\frac{2}{b^\ell} \sum_{1\leq h_1< h_2\leq \ell}\sum_{\substack{b^{h_1-1}\leq d_1 <b^{h_1}\\b^{h_2-1}\leq d_2 <b^{h_2} \\ d_2 \equiv d_1 \pmod{b^{h_1}} }}  \1_{\mathbb{P}}(d_1)\1_{\mathbb{P}}(d_2)
\sum_{\substack{n< b^\ell\\ n\equiv d_2 \pmod{b^{h_2}} }} 1,
 \end{align*}
 where we note that if $b^{h-1} \le d_1 < d_2 < b^h$ for some $h$, then we cannot have $n \equiv d_2 \equiv d_1 \pmod{b^{h-1}}$. On evaluating the sum over $n$ and applying Lemma \ref{lem:11}, as $b \rightarrow \infty$ this is equal to
\begin{align*}
& \langle\T_b\rangle_{b^\ell}  +\frac{2b}{\phi(b)\log^2 b} \sum_{1\leq h_1< h_2\leq \ell} \frac{1}{h_1h_2}\left(1+O\left( \frac 1{h_1 \log b} + \frac 1b + h_1^3h_2^2\frac{b^{\frac{h_1 - h_2}{2}}\phi(b)}{b}\log^5 b\right)\right)\\
=& \langle\T_b\rangle_{b^\ell}  +\frac{b}{\phi(b)\log^2 b} \left[\left(\sum_{h=1}^\ell \frac{1}{h}\right)^2-\sum_{h=1}^\ell \frac{1}{h^2}\right]+O\Bigg(\frac{\log \ell \log \log b}{\log^3 b} + \frac{\ell^4 \log^3 b }{\sqrt b}\Bigg),
\end{align*}
where we have applied the bound $\frac{b}{\phi(b)}\ll \log \log b$.

Combining this computation with  \eqref{eq:varnf} and Lemma \ref{lem:meannumber} implies the theorem.
\end{proof}

 Now assume that $b$ is fixed and $\ell \rightarrow \infty$. In this regime we have difficulty applying Lemma \ref{lem:11}, as the error term overwhelms the main term. However, much of the error term (in particular, that arising from the application of the Generalized Riemann Hypothesis) is expected to oscillate, so we might conjecture that it will not ultimately contribute to the variance. An examination of the proof of Lemma \ref{lem:11} suggests that for large $h_1$, we should have that the sum in Lemma \ref{lem:11} is given by $\frac{b^{h_2+1}}{\phi(b)h_1h_2\log^2b}\left(1-\frac 1b\right)^2 \left(1 + O\left(\frac 1{h_1\log b}\right)\right)$, which we can use to derive a conjecture for the variance in the large-$\ell$ limit. Precisely, we expect
 \begin{align*}
\frac{1}{b^\ell} \sum_{n< b^\ell}\T_b(n)^2 = & \langle\T_b\rangle_{b^\ell} +\frac{2b}{\phi(b)\log^2 b}\left(1 -\frac{1}{b}\right)^2 \sum_{1\leq h_1< h_2\leq \ell}\frac{1}{h_1h_2} \left(1 + O\left( \frac 1{h_1\log b}\right) \right)\\
= & \langle\T_b\rangle_{b^\ell} +\frac{b}{\phi(b)\log^2 b}\left(1 -\frac{1}{b}\right)^2 \log^2\ell+O_b \left(\log \ell \right),
\end{align*}
leading us to formulate Conjecture \ref{conj:varnf} after an application of Theorem \ref{thm:averagenf}.

The conjectured variance in the large-$\ell$ limit looks qualitatively different than what is obtained in Theorem \ref{thm:varnf}. Indeed, the term dominating the variance when $b \rightarrow \infty$ is $\asymp \frac{\log \ell }{\log b}$, a term that gets absorbed in the error term when $\ell \rightarrow \infty$. Conversely, the main term in Conjecture \ref{conj:varnf} is $\asymp \frac{1}{\log^2 b}\left(1 -\frac{1}{b}\right)^2\left(\frac{b}{\phi(b)}-1\right)\log^2\ell$, which is a secondary term when $b\rightarrow \infty$. This difference appears to be due to a clustering of prime truncations when $n$ is not coprime to the base $b$. That is, an integer $n$ has no prime truncations (except perhaps for one one-digit prime truncation) whenever $\gcd(n,b) > 1$, so that all prime truncations ``cluster'' among the remaining $n$ with $\gcd(n,b) = 1$.  Each $n$ with no prime truncations contributes $A^2$ to the variance, where $A$ is the average. As $b \to \infty$, the average $A$ is $\asymp \frac 1{\log b}$, which approaches $0$, so that the $A^2$ contributions are vanishingly small. Thus in the large-$b$ limit, this effect disappears. But, if $b$ is fixed and $\ell \to \infty$, the average does not approach $0$, so that this effect should contribute to a term in the variance of size $\asymp \log^2 \ell$.

One can then ask if the variance in the large-$\ell$ limit would be substantially smaller if restricted to those $n$ that are coprime to $b$. In this case the average number of truncations is
\begin{equation*}
\langle\T_b\rangle_{b^\ell,0} := \frac{1}{\phi(b) b^{\ell-1}}\sum_{\substack{n< b^\ell\\ \gcd(n,b)=1}}\T_b(n)=\frac{b}{\phi(b)\log b}\left(1-\frac{1}{b}\right)\log \ell+O_b(1),
\end{equation*}
whereas the same heuristic leading to Conjecture \ref{conj:varnf} suggests that the variance should be
\[\mathrm{Var}^0_{b^\ell}(\T_b):= \frac{1}{\phi(b)b^{\ell-1}} \sum_{\substack{n< b^\ell\\(n,b)=1}} \left(\T_b(n)-\langle\T_b\rangle_{b^\ell,0}\right)^2 \ll_b \log \ell.\]
Indeed the variance in this restricted setting more closely matches the variance as $b \to \infty$.

\subsection{Proof of Lemma \ref{lem:11} } \label{sec:proofoflemmanf}

To prove Lemma \ref{lem:11}, we follow the standard approach of first estimating the same sum with von Mangoldt weights.

Recall that
\[\psi(x):=\sum_{1 \le d \le x}\Lambda(d) \qquad \mbox{and} \qquad  \psi(\chi, x):=\sum_{1 \le d\leq x}\chi(d)\Lambda(d)\]
are the second Chebyshev and twisted second Chebyshev functions respectively.
The Generalized Riemann Hypothesis implies the estimates that 
\begin{equation}\label{eq:psi and psi chi}
\psi(x)=x+O\left(\sqrt{x}\log^2x\right)\quad \text{ and }\quad \psi(\chi, x)\ll x^\frac{1}{2}\log (Cx) \log(x),
\end{equation}
where $C$ is the conductor of $\chi$. The implied constant in \eqref{eq:psi and psi chi} is independent of $C$
(see \cite[Theorem 4.2]{Lee}). 

We start by proving the following result.
\begin{lem}\label{lem:lambdalambdanf} Let $0<h_1<h_2$ be integers and assume the Generalized Riemann Hypothesis. Let $b \ge 2$ and let $y_1, y_2$ be real numbers with $b^{h_1-1} \le y_1 \le b^{h_1}$ and $b^{h_2-1} \le y_2 \le b^{h_2}$. As $b\rightarrow \infty$, we have
\begin{equation*}
\sum_{\substack{b^{h_1-1}\leq d_1< y_1 \\ b^{h_2-1}\leq d_2< y_2\\d_2\equiv d_1\pmod{b^{h_1}}}}\Lambda(d_1)\Lambda(d_2) = \frac{(y_1-b^{h_1-1})(y_2-b^{h_2-1})}{\phi(b^{h_1})} + O\left(\sqrt{y_1y_2} h_1^2 \log^2b\left(\frac{\sqrt{y_2}}{\phi(b^{h_1})} +  h_2^2 \log^2b \right)\right).
\end{equation*}
\end{lem}
\begin{proof} 
By the orthogonality of Dirichlet characters mod $b^{h_1}$, we have
\begin{align*}
&\sum_{\substack{b^{h_1-1}\leq d_1< y_1\\ b^{h_2-1}\leq d_2< y_2\\d_2\equiv d_1\pmod{b^{h_1}}}}\Lambda(d_1)\Lambda(d_2)=\frac{1}{\phi(b^{h_1})}\sum_{\chi\pmod{b^{h_1}}}\sum_{\substack{b^{h_1-1}\leq d_1< y_1\\ b^{h_2-1}\leq d_2< y_2}}\chi(d_1)\overline{\chi}(d_2)\Lambda(d_1)\Lambda(d_2)\\
&=\frac{1}{\phi(b^{h_1})}(\psi(y_1)-\psi(b^{h_1-1}))(\psi(y_2)-\psi(b^{h_2-1}))\\
&\qquad +\frac{1}{\phi(b^{h_1})}\sum_{\substack{\chi\pmod{b^{h_1}}\\\chi\not = \chi_0}}
(\psi(\chi, y_1)-\psi(\chi, b^{h_1-1}))(\psi(\overline{\chi}, y_2)-\psi(\overline{\chi}, b^{h_2-1})),
\end{align*}
where $\chi\pmod{b^{h_1}}$ indicates that the sum is being taken over all the Dirichlet characters modulo $b^{h_1}$, and $\chi_0$ denotes the principal character.

By applying the estimates in \eqref{eq:psi and psi chi}, we can write this as
\begin{align*}
&\frac{(y_1 - b^{h_1-1})(y_2-b^{h_2-1})}{\phi(b^{h_1})}+O\left(\frac{\sqrt{y_1} y_2 h_1^2 \log^2 b}{\phi(b^{h_1})}\right)+O\left(\sqrt{y_1y_2}h_1^2h_2^2\log^4b\right),
\end{align*}
which completes the proof.
\end{proof}

We now turn to the proof of Lemma \ref{lem:11}, whose statement differs from that of Lemma \ref{lem:lambdalambdanf} only in the aspect that $\Lambda$ is replaced by the indicator function of primes $\1_{\mathcal P}$. To go from $\Lambda$ to $\1_{\mathcal P}$, we must remove the weighting by the degree as well as the contribution from higher prime powers, and replace $y_1$ and $y_2$ respectively with $b^{h_1}$ and $b^{h_2}$. That is, we have
\begin{align*}
\sum_{\substack{b^{h_1-1} \le d_1 < b^{h_1} \\ b^{h_2 - 1} \le d_2 < b^{h_2} \\ d_2 \equiv d_1 \pmod{b^{h_1}}}} \1_{\mathcal P}(d_1)\1_{\mathcal P}(d_2) &= \sum_{\substack{b^{h_1-1} \le d_1 < b^{h_1} \\ b^{h_2 - 1} \le d_2 < b^{h_2} \\ d_2 \equiv d_1 \pmod{b^{h_1}}}} \frac{\Lambda(d_1)\Lambda(d_2)}{\log d_1\log d_2} - \sum_{\substack{b^{h_1-1} \le d_1 < b^{h_1} \\ b^{h_2 - 1} \le d_2 < b^{h_2} \\ d_2 \equiv d_1 \pmod{b^{h_1}} \\ d_1 \text{ or }d_2 = h^k,\, k \ge 2}} \frac{\Lambda(d_1)\Lambda(d_2)}{\log d_1\log d_2}.
\end{align*}
We begin with the first term on the right, which we can estimate using Lemma \ref{lem:lambdalambdanf} and partial summation. 
Precisely, we define
\begin{align*}
a_{d_1,d_2} := \1_{\substack{b^{h_1-1} \le d_1 < b^{h_1} \\ b^{h_2-1} \le d_2 < b^{h_2} \\ d_2 \equiv d_1 \pmod{b^{h_1}}}} \Lambda(d_1) \Lambda(d_2), \quad c_{d_1,d_2} := \frac{a_{d_1,d_2}}{\log d_2}, \quad f_{d_1,d_2} := \frac{c_{d_1,d_2}}{\log d_1},
\end{align*}
as well as the summation functions
\begin{equation*}
A(y_1,y_2) := \sum_{\substack{d_1 < y_1 \\ d_2 < y_2}} a_{d_1,d_2}, \quad C(y_1,y_2) := \sum_{\substack{d_1 < y_1 \\ d_2 < y_2}} c_{d_1,d_2}, \quad F(y_1,y_2) := \sum_{\substack{d_1 < y_1 \\ d_2 < y_2}} f_{d_1,d_2}.
\end{equation*}
Note that these functions are only nonzero if $b^{h_i - 1} < y_i \le b^{h_i}$, and that Lemma \ref{lem:lambdalambdanf} provides an asymptotic evaluation of $A(y_1,y_2)$. Our goal is to compute $F(b^{h_1},b^{h_2})$.

By applying partial summation to the sum over $d_2$, we get that for any $y_1$,
\begin{align*}
C(y_1,b^{h_2}) &= \frac{A(y_1,b^{h_2})}{\log b^{h_2}} + \int_{b^{h_2-1}}^{b^{h_2}} \frac{A(y_1,u)}{u \log^2 u} \mathrm du \\
&= \frac{A(y_1,b^{h_2})}{h_2 \log b} + \frac{(y_1-b^{h_1-1})}{\phi(b^{h_1})}\int_{b^{h_2-1}}^{b^{h_2}} \frac{(u-b^{h_2-1})}{u \log^2 u} \mathrm du  \\
&+ O\left(\sqrt{y_1}h_1^2 \log^2 b \left(\frac 1{\phi(b^{h_1})} \int_{b^{h_2-1}}^{b^{h_2}} \frac{\mathrm du}{\log^2 u} + h_2^2\log^2 b \int_{b^{h_2-1}}^{b^{h_2}} \frac{\mathrm du}{\sqrt u \log^2 u} \right)\right) \\
&= \frac{(y_1-b^{h_1-1})}{\phi(b^{h_1})}\left( \frac{b^{h_2}- b^{h_2-1}}{h_2 \log b} + \mathrm{li}_2(b^{h_2-1},b^{h_2})  - \frac{b^{h_2-1}}{h_2(h_2-1)\log b} \right) \\
&+ O\left(\sqrt{y_1}h_1^2\log^2 b\left(\frac{b^{h_2}}{\phi(b^{h_1}) h_2 \log b} + b^{\frac{h_2}{2}} h_2 \log b  \right) \right),
\end{align*}
where $\mathrm{li}_2(x_1,x_2) := \int_{x_1}^{x_2} \frac{\mathrm{d}u}{\log^2u}$ is the second-order logarithmic integral.

Another application of partial summation to the sum over $d_1$ implies that
\begin{align*}
F&(b^{h_1},b^{h_2}) = \frac{C(b^{h_1},b^{h_2})}{\log b^{h_1}} + \int_{b^{h_1-1}}^{b^{h_1}} \frac{C(u,b^{h_2})}{u \log^2 u} \mathrm du \\
&= \left( \frac{b^{h_2}- b^{h_2-1}}{h_2 \log b} + \mathrm{li}_2(b^{h_2-1},b^{h_2})  -  \frac{b^{h_2-1}}{h_2(h_2-1)\log b} \right) \left(\frac{(b^{h_1}-b^{h_1-1})}{\phi(b^{h_1})h_1\log b} + \int_{b^{h_1-1}}^{b^{h_1}} \frac{(u-b^{h_1-1})}{\phi(b^{h_1})u \log^2 u } \mathrm du\right) \\
&+ O\left(\frac{b^{h_2+\frac{h_1}{2}} h_1^2 \log b}{\phi(b^{h_1})h_2} + b^{\frac{h_1+h_2}2} h_1^2h_2\log^3b + \left(\frac{b^{h_2}h_1^2\log b}{\phi(b^{h_1})h_2} + b^{\frac{h_2}2} h_1^2h_2\log^3b\right)\int_{b^{h_1-1}}^{b^{h_1}} \frac{\mathrm{d}u}{\sqrt u \log^2u} \right),
\end{align*}
which then simplifies to
\begin{align*}
&\frac 1{\phi(b^{h_1})}\left( \frac{b^{h_2}- b^{h_2-1}}{h_2 \log b} + \mathrm{li}_2(b^{h_2-1},b^{h_2})  -  \frac{b^{h_2-1}}{h_2(h_2-1)\log b} \right)\left( \frac{b^{h_1}- b^{h_1-1}}{h_1 \log b} + \mathrm{li}_2(b^{h_1-1},b^{h_1})  -  \frac{b^{h_1-1}}{h_1(h_1-1)\log b} \right) \\
&+ O\left(\frac{b^{h_2 - \frac{h_1}{2}+1}h_1^2\log b}{\phi(b)h_2} + b^{\frac{h_1+h_2}2} h_1^2 h_2 \log^3 b\right).
\end{align*}
Note that if $h_1 = 1$, this expression is ill-defined; in that case, we can start the range of consideration at $d_1 = 2$ instead of $d_1 = 1 = b^{h_1-1}$, and the term $\frac{b^{h_1-1}}{h_1(h_1-1)\log b}$ should be replaced with $O(1)$.

The error term we obtain here closely mirrors the results obtained in the polynomial case in Section \ref{sec:the polynomial setting}. It is stronger than what we need here, and stronger than is stated in Lemma \ref{lem:11}. Taking only the dominating term in each product as $b \to \infty$ and regrouping error terms completes the proof of Lemma \ref{lem:11}.

\section{A heuristic for the maximum number of prime truncations in the integer setting}\label{sec:maxnf}

In this section we present a heuristic for the maximum value of the number of irreducible truncations of an $\ell$-digit number in base $b$. We will make crucial use of the \emph{Borel--Cantelli lemmas}, which we recall here for convenience.

\begin{thm}[Borel--Cantelli]\label{thm:BC}
Let $(E_n)_{n \ge 1}$ be a sequence of events in a probability space. If
$
\sum_{n=1}^{\infty} \mathbb{P}(E_n) < \infty,
$
then the probability that infinitely many of the events occur is zero; that is,
$
\mathbb{P}\!\left( \limsup_{n \to \infty} E_n \right) = 0.
$

Moreover, if the events $E_n$ are pairwise independent and
$\sum_{n=1}^\infty \mathbb P(E_n) \to \infty,$
then the probability that infinitely many of the events occur is $1$.
\end{thm}

We consider the probability that an $\ell$-digit base-$b$ number has at least $k$ irreducible truncations, which we denote $p_{b,\ell}(k)$. Precisely, we define
\[p_{b,\ell}(k):=\mathbb{P}(n \in [b^{\ell-1},b^\ell)\, : \,\T_b(n)\geq k).\]

For a single $\ell$-digit number $n$, the probability that $n$  has strictly fewer than $\ell$ truncations is $(1-p_{b,\ell}(\ell))$. If
\begin{equation}\label{eq:large-b-limit-maximum-heuristic}
\sum_{b=2}^\infty(1-p_{b,\ell}(\ell))^{\#[b^{\ell-1},b^\ell)}<\infty,
\end{equation}
then, Theorem \ref{thm:BC} implies the following expected fact.
\begin{fact}\label{fact:nfq} For each fixed $\ell > 1$, there are only finitely many $b$'s such that no $\ell$-digit base-$b$ number $n$ has $\ell$ prime truncations.
\end{fact}
As we will see below, we expect \eqref{eq:large-b-limit-maximum-heuristic} to hold, and thus also Expected Fact \ref{fact:nfq}.

We now consider the setting where $b$ is fixed and $\ell\rightarrow \infty$. The maximum number of prime truncations of an $\ell$-digit base-$b$ number should be asymptotic to $K(\ell)$, where $K(\ell)$ is a function of $\ell$ such that
\begin{equation}\label{eq:nf-maximum-as-ell-goes-to-infinity-at-least-cKell-truncations}
\sum_{\ell=1}^{\infty} \#\{n\in [b^{\ell-1},b^\ell) : \T_b(n) \geq cK(\ell)\}
\end{equation}
converges whenever $c > 1$ but does not converge whenever $c \le 1$.
We will provide heuristic arguments in support of the following expected fact.
\begin{fact}\label{fact:nfn}  Let $b>1$ a fixed integer and let $K(\ell)$ be as above. Then, as $\ell\rightarrow \infty$,
\[K(\ell)\asymp \frac{\ell \log b}{\log \ell}.\]
\end{fact}
In support of Expected Facts \ref{fact:nfq} and \ref{fact:nfn}, we proceed to heuristically estimate $p_{b,\ell}(k)$. By definition of $p_{b,\ell}(k)$, \eqref{eq:nf-maximum-as-ell-goes-to-infinity-at-least-cKell-truncations} is equal to $\sum_{\ell =1}^\infty (b-1)b^{\ell-1} p_{b,\ell}(cK(\ell))$, so that estimating $p_{b, \ell}(k)$ suffices.

We begin by considering the auxiliary probabilities
\[p^{(0)}_{b,\ell}(k):=\mathbb{P}(n \in [b^{\ell-1},b^\ell)\, : \,\T_b(n)= k) \mbox{ and } p^{(1)}_{b,\ell}(k):=\sum_{h=k}^{\ell}\binom{h}{k}p^{(0)}_\ell(h),\]
so that $p^{(0)}_{b,\ell}(k)$ is the probability that an $\ell$-digit number in base $b$ has precisely $k$ prime truncations, and $(b-1)b^{\ell-1}p^{(1)}_{b,\ell}(k)$ counts the number of pairs $(n,J)$ such that $n$ has $\ell$ digits in base $b$, $J\subseteq [1,\ell]$ has size $k$, and all $k$ truncations of $n$ represented by the set $J$ are prime.
 Then $p_{b,\ell}(k)$, $p_{b,\ell}^{(0)}(k)$, and $p_{b,\ell}^{(1)}(k)$ are related via
\[p_{b,\ell}(k)=\sum_{j=k}^\ell p^{(0)}_{b,\ell}(j) \quad \text{and} \quad p^{(0)}_{b,\ell}(j)=\sum_{h=j}^\ell (-1)^{h-j} \binom{h}{j} p^{(1)}_{b,\ell}(h),\]
so that we further have
\begin{equation*}
p_{b,\ell}(k)= \sum_{j=k}^\ell \sum_{h=j}^\ell (-1)^{h-j} \binom{h}{j}p^{(1)}_{b,\ell}(h) \quad \text{and} \quad p_{b,\ell}(k)=\sum_{h=k}^\ell (-1)^{h-k} \binom{h-1}{k-1}p^{(1)}_{b,\ell}(h)\end{equation*}
for $k\geq 1$ and $p_{b,\ell}(0)=1$. Note also that $p_{b,\ell}(\ell) = p_{b,\ell}^{(0)}(\ell) = p_{b,\ell}^{(1)}(\ell)$. 

We proceed by using the Cram\'er random model for the prime numbers to estimate $p_{b,\ell}^{(1)}(k)$: namely, modeling the indicator function of primes by a sequence $(X_n)_{n \ge 1}$ of independent random variables where $X_n = 1$ with probability $\frac{1}{\log n}$ and $0$ otherwise. For the purposes of this heuristic it will be simpler for us to use a modification where for any $n$ with $d$ digits in base $b$, the probability that $X_n = 1$ is $\frac{1}{\log b^d}$; this approximation does not change the predicted result. 

Let $[(n)_b]_{(k)}$ denote the $k$th digit of $n$ in base $b$. Under the Cram\'er model we have
\begin{align}\label{eq:p1-nf}
p^{(1)}_{b,\ell}(h)\approx&\frac{1}{ (b-1)b^{\ell-1}}\sum_{b^{\ell-1}\leq n <b^\ell} \Bigg(\sum_{\substack{J \subset [1,\ell-1] \\ |J| = h \\ [(n)_b]_{(d)} \ne 0\, \forall d \in J}} \prod_{d \in J} \frac{1}{d\log b} + \frac {1}{\ell\log b} \sum_{\substack{J \subset [1,\ell-1] \\ |J| = h-1 \\ [(n)_b]_{(d)} \ne 0\, \forall d \in J}} \prod_{d \in J} \frac{1}{d\log b} \Bigg)\nonumber\\
 \approx&\sum_{\substack{J \subset [1,\ell-1] \\ |J| = h }}\left(\frac{b-1}{b}\right)^h \frac{1}{\log^h b}\prod_{d \in J} \frac{1}{d} + \frac {1}{\ell\log b} \sum_{\substack{J \subset [1,\ell-1] \\ |J| = h-1 }}\left(\frac{b-1}{b}\right)^{h-1}  \frac{1}{\log^{h-1} b}\prod_{d \in J} \frac{1}{d}.
\end{align}
As $b \to \infty$, $p_{b,\ell}(\ell)$ is given by
\[p_{b,\ell}(\ell)=p^{(1)}_{b,\ell}(\ell)\approx \frac{1}{\log^{\ell} b}\left(\frac{b-1}{b}\right)^{\ell-1} \frac{1}{\ell!},\]
so that we expect
\[\sum_{b=2}^\infty(1-p_{b,\ell}(\ell))^{\#[b^{\ell-1},b^\ell)}\ll \sum_{b=2}^\infty \left[\left(1-\frac{1}{\ell! b}\right)^{\ell!b}\right]^{\frac{(1-b^{-1})b^{\ell-1}}{\ell!}}\ll \sum_{b=2}^\infty \exp\left(-\frac{b^{\ell-1}}{\ell!}\right) \ll 1.\]
Via Theorem \ref{thm:BC}, this supports Expected Fact \ref{fact:nfq}.

For large $\ell$ and for $h$ much smaller than $\ell$, the first term in \eqref{eq:p1-nf} dominates, so that
\[p^{(1)}_{b,\ell}(h)\approx \left(\frac{b-1}{b}\right)^h  \frac{1}{\log^{h} b}\sum_{0<d_1<\cdots <d_h<\ell }\frac{1}{d_1\cdots d_h},\]
and thus
\[p_{b,\ell}(k)\approx \sum_{h=k}^\ell (-1)^{h-k} \binom{h-1}{k-1}\left(\frac{b-1}{b}\right)^h  \frac{1}{\log^{h} b} \sum_{0<d_1<\cdots <d_h<\ell }\frac{1}{d_1\cdots d_h}.\]
We have
\begin{equation}\label{eq:nf}
\sum_{0<d_1<\cdots <d_h<\ell}\frac{1}{d_1\cdots d_{h}}=[\mathcal{F}_\ell(x)]_{(h)} \quad \text{for} \quad \mathcal{F}_\ell(x):=\frac{\Gamma(\ell+x)}{\Gamma(\ell)\Gamma(1+x)},
\end{equation}
where in general $[\mathcal{F}(x)]_{(h)}$ denotes the coefficient of $x^h$ in the power series for $\mathcal{F}(x)$. Setting $\omega_b =\left(\frac{b-1}{b}\right)  \frac{1}{\log b}$, we have
\begin{align*}
 p_{b,\ell}(k)\approx &\omega_b^k \sum_{h=k}^\ell \binom{h-1}{k-1}(-\omega_b)^{h-k}
[\mathcal{F}_\ell(x)]_{(h)}
= \frac{\omega_b^k}{(k-1)!}\frac{d^{k-1}}{dx^{k-1}}\left.\frac{\mathcal{F}_\ell(x)-1}{x}\right|_{x=-\omega_b}.
\end{align*}
Notice that
\begin{align*}\mathcal{F}_\ell (x)=\frac{\Gamma(\ell+x)}{\Gamma(\ell)\Gamma(1+x)}=&\frac{\ell^x}{\Gamma(1+x)}\left(1+O\left(
\frac{1}{\ell}\right)\right).
\end{align*}

Differentiating to simplify the estimate for $p_{b,\ell}(k)$ then yields
\begin{align*}
 p_{b,\ell}(k)\approx
&\frac{\omega_b^k}{(k-1)!} \ell^{-\omega_b} \sum_{j=0}^{k-1}\binom{k-1}{j}\log^{k-1-j} \ell \frac{d^{j}}{dx^{j}}\left.\frac{1}{x\Gamma(1+x)}\right|_{x=-\omega_b}+1.
\end{align*}
Since
\[\frac{1}{x\Gamma(1+x)}
   = \frac{1}{x}
     \exp\left(
        \gamma x
        + \sum_{n= 2}^\infty \frac{(-1)^{n+1}}{n}\zeta(n) x^{n}
     \right)\]
for $|x|<1$, and since $0<\omega_b<1$, we use the approximation
\[\left. \frac{d^{j}}{dx^{j}}\frac{1}{x\Gamma(1+x)}\right|_{x=-\omega_b}\approx -\frac{j!}{\omega_b^{j+1}}.\]

This implies that
\begin{align*}
 p_{b,\ell}(k)\approx &-\ell^{-\omega_b} \sum_{j=0}^{k-1}\frac{(\omega_b \log \ell)^{k-1-j} }{(k-1-j)!}+1=\ell^{-\omega_b} \sum_{j=k}^{\infty}\frac{(\omega_b \log \ell)^{j} }{j!}.
\end{align*}
As long as $k\geq \log^2\ell$, we expect the term with $j=k$ to dominate and thus
\begin{align*}
 p_{b,\ell}(k)\approx \ell^{-\omega_b} \frac{(\omega_b \log \ell)^{k} }{k!}.
\end{align*}

Following the discussion of the application of Theorem \ref{thm:BC} that led to Expected Fact \ref{fact:nfn}, we want to choose  $k=K(\ell)$ such that
\[\frac{(\omega_b \log \ell)^{K(\ell)} }{\ell^{\omega_b} K(\ell)!}(b^\ell-b^{\ell-1})\approx \frac{1}{\ell}.\]
Applying Stirling's approximation formula, we obtain
\[\left(\frac{e\omega_b \log \ell}{K(\ell)}\right)^{K(\ell)}\approx \frac{\ell^{\omega_b-1}\sqrt{K(\ell)}}{b^\ell},\]
where we have ignored constants.
By applying logarithms on both sides, we obtain
\[K(\ell)(1+\log \omega_b +\log \log \ell -\log K(\ell))\asymp \frac{1}{2}\log K(\ell)+(\omega_b-1)\log \ell -\ell \log b.\]
We discard the term  $(\omega_b-1)\log \ell$, as the term $\ell \log b$ will dominate the right-hand side above, as well as $\frac{1}{2}\log K(\ell)$, since we expect $K(\ell)$ not to be larger than $\ell$. This leads to
\[K(\ell)\asymp  \frac{\ell \log b}{W\left(\frac{\ell \log b}{e \omega_b \log \ell }\right)},\]
where $W(x)$ is the Lambert function: namely, the solution to $x=W(x)e^{W(x)}$. Approximating $W(x)$ by $\log x$, we finally obtain
$K(\ell)\asymp \frac{\ell \log b}{\log \ell},$
in support of Expected Fact \ref{fact:nfn}.

\section{The polynomial setting}\label{sec:the polynomial setting}
We now discuss an analogous problem in the setting of polynomials over a finite field.  Let $b(T)\in \F_q[T]$ of degree $m$, so that any polynomial $g \in \mathbb F_q[T]$ has $\lfloor \frac{\deg(g)}{m}\rfloor+1$ terms in its base-$b(T)$ expansion. For any $f \in \F_q[T]$,
define $\T_{q,b}(f)$ to be the number of irreducible truncations of $f(T)$ in base $b(T)$; namely,
\[\T_{q,b}(f):=\sum_{\substack{g\in \F_q[T]\\ \deg(g)\leq \deg(f) \\ g\equiv f \pmod{b^{\lfloor \frac{\deg(g)}{m}\rfloor+1}}}}\1_\mathcal{P}(g),\]
where $\1_\mathcal{P}$ denotes the indicator function for irreducible (not necessarily monic) polynomials in $\F_q[T]$. The congruence condition below the sum implies that the last digits of $f(T)$ coincide with $g(T)$, so that $g$ really is a truncation of $f$. 

We are interested in the average number of irreducible truncations over all polynomials with at most $\ell$ digits in base $b(T)$, defined in \eqref{eq:defn-average-truncations-in-polynomial-case}. Our first goal is to prove Theorem \ref{thm:averageff}, which is a combination of Theorem \ref{thm:preliminary-average-polynomials} and Lemma \ref{lem:rewriting the main term for the average in function fields} below.
\begin{thm}\label{thm:preliminary-average-polynomials}
We have that
\begin{equation}\label{eq:asymptotic-all-limits-for-average-truncation-polynomials}
\langle \T_{q,b} \rangle_{b^\ell} = (q-1) \sum_{t=1}^m q^{-t} \sum_{\substack{1 \le h < m \ell \\ h \equiv -t \pmod m}} \frac 1h + E,
\end{equation}
where $E$ is an error term such that
\begin{equation*}
E \ll \begin{cases} \frac 1q  &\text{ as } q \to \infty \text{ or as } \ell \to \infty, \text{ and} \\
\frac{q \log^2(m\ell)}{q^{\frac{m}{2}}} &\text{ as } m \to \infty.
\end{cases}
\end{equation*}
\end{thm}

In each of the limits as $q,m$, or $\ell \to \infty$, the main term can be simplified, which we record in the following lemma.
\begin{lem}\label{lem:rewriting the main term for the average in function fields}
The main term in the right-hand side of \eqref{eq:asymptotic-all-limits-for-average-truncation-polynomials} is asymptotic to
\begin{equation*}
(q-1) \sum_{t=1}^m q^{-t} \sum_{\substack{1 \le h < m \ell \\ h \equiv -t \pmod m}} \frac 1h = 
\begin{cases}
\displaystyle \sum_{\substack{1 \le h < m \ell \\ h \equiv -1 \pmod m}} \frac 1h + O\Big(\frac{\log \ell}{mq}\Big) &\text{ as } q\to \infty, \\
\displaystyle \frac 1m \sum_{h=1}^\ell \frac 1h + \frac{q}{(q-1)m^2} \sum_{h=1}^\ell \frac 1{h^2} + O\Big(\frac 1{m^3}\Big) &\text{ as } m \to \infty, \\
\displaystyle \frac{1}{m}\Big(1-\frac{1}{q^{m}}\Big) \log \ell + O\Big(\frac 1m\Big) &\text{ as } \ell \to \infty.
\end{cases}
\end{equation*}
\end{lem}
Recall that the implied constants in all error terms are independent of $\ell, m$, and $q$.
\begin{rem}
Of particular interest is the case where $m = 1$ and $b(T) = T$. In this case the analogy between the polynomial case and the integer case is perhaps more intuitive: the ``digits'' of a polynomial $f \in \mathbb F_q[T]$ are merely its coefficients. If $m = 1$, many of the estimates in both Theorem \ref{thm:preliminary-average-polynomials} and Lemma \ref{lem:rewriting the main term for the average in function fields} simplify. Notably, in the large $q$ limit, we get
\begin{equation*}
\langle \T_{q,T} \rangle_{T^\ell} = \frac{q-1}{q}\sum_{\substack{1 \le h < \ell}} \frac 1h + O\left(\frac{1}{q}\right),
\end{equation*}
where the implied constant in the error does not depend on $\ell$. 
\end{rem}

We begin by proving Theorem \ref{thm:preliminary-average-polynomials}.
\begin{proof}[Proof of Theorem \ref{thm:preliminary-average-polynomials}]
By expanding the definition of $\langle \T_{q,b}\rangle_{b^\ell}$ and swapping the order of summation, we get
\begin{align*}
 \langle \T_{q,b}\rangle_{b^\ell}
 =&\frac{1}{q^{m\ell}}\sum_{\substack{g\in \F_q[T]\\\deg(g)<m \ell}} \1_\mathcal{P}(g)
 \sum_{\substack{f\in \F_q[T]\\ \deg(f)<m\ell\\
 f\equiv g \pmod{b^{\lfloor\frac{\deg(g)}{m}\rfloor+1 }}}}1.
\end{align*}
If $g$ has $k$ digits in base $b$, then the inside sum is given by $q^{m(\ell-k)}$. Separating the sum over $g$ by the number of digits of $g$ in base $b$ (or equivalently, by $\lfloor \frac{\deg(g)}{m}\rfloor+1$), and simplifying, we have
\begin{align*}
 \langle \T_{q,b}\rangle_{b^\ell}
  =&
\sum_{k=1}^{\ell} q^{-mk}
 \sum_{\substack{g\in \F_q[T]\\ \lfloor \frac{\deg(g)}{m}\rfloor= k-1 }} \1_\mathcal{P}(g).
 \end{align*}

 The number of monic irreducible polynomials in $\mathbb F_q[T]$ of degree $d$ is given by $\frac 1d \sum_{\ell|d} \mu(\ell)q^{\frac{d}{\ell}}$, so that, noting that the sums are over all polynomials and not just monic ones, we have
\begin{align*}
 \langle \T_{q,b}\rangle_{b^\ell}=&
\sum_{k=1}^{\ell} \frac{q-1}{q^{mk}} \sum_{\substack{h=m(k-1)\\h>0}}^{mk-1} \frac{1}{h}\sum_{j\mid h}\mu(j)q^{\frac{h}{j}} =(q-1)
\sum_{k=1}^{\ell} q^{-mk} \sum_{\substack{h=m(k-1)\\h>0}}^{mk-1} \frac{q^{h}}{h}
+E,
 \end{align*}
 where
 \begin{equation*}
E = (q-1)\sum_{k=1}^{\ell} q^{-mk} \sum_{\substack{h=m(k-1)\\h>0}}^{mk-1} \frac{1}{h}\sum_{\substack{j\mid h\\j>1}}\mu(j)q^{\frac{h}{j}}.
 \end{equation*}

The main term can be rewritten as
\begin{equation*}
(q-1)\sum_{k=1}^{\ell} q^{-mk} \sum_{\substack{h=m(k-1)\\h>0}}^{mk-1} \frac{q^{h}}{h}=
(q-1)\sum_{\substack{h=1}}^{m\ell-1} \frac{q^{h}}{hq^{(\lfloor \frac{h}{m}\rfloor +1)m}}
= (q-1)\sum_{t=1}^m q^{-t} \sum_{\substack{1 \le h < m\ell \\ h \equiv -t \pmod m}} \frac 1h,
\end{equation*}
which is precisely the form that appears in the theorem statement. It remains to estimate $E$. This involves slightly different arguments depending on the bound being taken. To begin with, we rewrite this sum by swapping the order of summation and noting that $k = \lfloor \frac hm \rfloor + 1$, as well as writing $h = jr$, to get that
\begin{align}
 (q-1)\sum_{k=1}^{\ell} q^{-mk} \sum_{\substack{h=m(k-1)\\h>0}}^{mk-1} \frac{1}{h}\sum_{\substack{j\mid h\\j>1}}\mu(j)q^{\frac{h}{j}}
=&  (q-1)\sum_{j=2}^{m\ell-1}\frac{ \mu(j)}{j }\sum_{r=1}^{\frac{m\ell-1}{j}} \frac{q^{r}}{rq^{(\lfloor \frac{rj}{m}\rfloor +1)m}}.\label{eq:first term in E}
\end{align}

We will use the bounds that for all $r,m \ge 1$ and $j \ge 2$,
\begin{equation}\label{eq:floor of rj over m plus one - lower bound}
\lfloor rj/m\rfloor +1\geq \max\left\{\frac{rj}{m}+\frac{1}{m}, 1\right\},
\end{equation}
which implies that 
\begin{equation}\label{eq:floor of r minus m times rj over m plus one - upper bound}
r - (\lfloor rj/m\rfloor + 1)m \le \min \left\{-1-r(j-1), r-m\right\} \le -m/2,
\end{equation}
where the last inequality follows because $j \ge 2$ and the minimum of $-1-r$ and $r-m$ is bounded by $-m/2$.

We can use \eqref{eq:floor of rj over m plus one - lower bound} to bound the absolute value of \eqref{eq:first term in E} via
\begin{equation*}
\leq \frac{q-1}{q}\sum_{j=2}^{\infty}\sum_{r=1}^{\infty}q^{r(1-j)}= \frac{q-1}{q} \sum_{j = 2}^\infty \frac{q^{1-j}}{1-q^{1-j}} \le \sum_{j = 2}^\infty q^{1-j}\leq \frac{1}{q-1} \ll \frac 1q,
\end{equation*}
which completes the argument as $q$ or $\ell \to \infty$. If $m \to \infty$ we can do somewhat better by applying \eqref{eq:floor of r minus m times rj over m plus one - upper bound} to get that
\begin{equation*}
 \Bigg|(q-1) \sum_{j=2}^{m\ell-1}\frac{ \mu(j)}{j }\sum_{r=1}^{\frac{m\ell-1}{j}} \frac{q^{r}}{rq^{(\lfloor \frac{rj}{m}\rfloor +1)m}}\Bigg|\leq  q^{1-\frac{m}{2}}\sum_{j=2}^{m\ell-1}\frac{1}{j }\sum_{r=1}^{\frac{m\ell-1}{j}} \frac{1}{r}
 \ll  q^{1-\frac{m}{2}}\log^2 (m\ell),
 \end{equation*}
 which completes the argument.
\end{proof}

We now turn to the proof of Lemma \ref{lem:rewriting the main term for the average in function fields}, in order to understand the asymptotic behavior of the main term in each of the three limits.
\begin{proof}[Proof of Lemma \ref{lem:rewriting the main term for the average in function fields}]
We proceed in general by applying the following identity for $1 \le t \le m-1$, which follows from taking the Taylor expansion of $\frac 1{1-\frac t{km}}$ for $k \ge 2$:
\begin{equation}\label{eq:harmonic series mod m nonzero congruence class}
 \sum_{\substack{1\leq h <m\ell\\ h\equiv -t\pmod{m}}}\frac{1}{h}= \frac 1{m-t} + \frac 1m \sum_{k=2}^\ell \frac 1k + \frac t{m^2} \sum_{k=2}^\ell \frac 1{k^2} + O\left(\frac{t^2}{m^3}\right),
\end{equation}
as well as the exact identity
\begin{equation}\label{eq:harmonic series divisible by m}
 \sum_{\substack{1\leq h <m\ell\\ h\equiv 0\pmod{m}}}\frac{1}{h}=\frac{1}{m}\sum_{k=1}^{\ell-1} \frac{1}{k}=\frac{1}{m}\sum_{k=1}^{\ell} \frac{1}{k}-\frac{1}{m\ell}.
 \end{equation} 

In order to handle the sums over $t$ of the terms $\frac 1{m-t}$ coming from \eqref{eq:harmonic series mod m nonzero congruence class}, we will make additional use of the following identity for $m \ge 2$:
\begin{equation}\label{eq:extra term one over m minus t}
\sum_{t=1}^{m-1} \frac 1{(m-t)q^t} = \frac 1m \sum_{t = 1}^{m/2} \frac 1{(1-\tfrac tm) q^t} + O\left(\frac{m}{q^{\frac{m}{2}}}\right).
\end{equation}
As $q \to \infty$, \eqref{eq:harmonic series mod m nonzero congruence class} applied to $2 \le t \le m - 1$ along with \eqref{eq:extra term one over m minus t} and \eqref{eq:harmonic series divisible by m} imply that all terms except for the $t = 1$ term contribute $\ll \frac{\log \ell}{mq}$, which implies the result. As $\ell \to \infty$, the proof is completed by applying \eqref{eq:harmonic series mod m nonzero congruence class} for all $1 \le t \le m-1$ as well as \eqref{eq:extra term one over m minus t} and \eqref{eq:harmonic series divisible by m} and observing that
\begin{equation}\label{eq:sumkqk}
\sum_{t=1}^{m} \frac{t}{q^t}
=
\frac{q}{(q-1)^2}
\left(
1 - \frac{m+1}{q^m} + \frac{m}{q^{m+1}}
\right).
\end{equation}
Finally as $m \to \infty$, we again apply \eqref{eq:harmonic series mod m nonzero congruence class} for all $1 \le t \le m-1$ as well as \eqref{eq:extra term one over m minus t} and \eqref{eq:harmonic series divisible by m}. An additional Taylor expansion of $\frac 1{1-\tfrac tm} = 1 + \frac tm + O(t^2/m^2)$ applied for $1 \le t\le \lfloor m/2 \rfloor$ in the right hand side of \eqref{eq:extra term one over m minus t} as well as some rearranging and \eqref{eq:sumkqk} completes the argument in this case as well.
\end{proof}

We next consider the variance, defined in \eqref{eq:variance-definition-gen}. To estimate the variance, we will need the following lemma, proven in Section \ref{sec:proofsuperLemma1.2}.
\begin{lem}\label{lem:correlations-using-random-matrix-theory-general}
Let $m, h_1,$ and $h_2$ be positive integers such that $\lfloor \frac{h_1}{m}\rfloor<\lfloor \frac{h_2}{m}\rfloor$. Then, as $q \rightarrow \infty$,
\begin{align*}
 \sum_{\substack{ g_1,g_2\in \mathbb{F}_q[T]\\
 \deg(g_i)=h_i \\g_2 \equiv g_1 \pmod{b^{\lfloor \frac{h_1}{m}\rfloor+1}}}}&\1_\mathcal{P}(g_1)\1_\mathcal{P}(g_2) =
 \frac{(q-1)^2q^{h_2+m\{\frac{h_1}{m}\}}}{\Phi(b)h_1h_2} + O\left(\frac{(h_1+m)^2}{h_1h_2} q^{\frac{h_1+h_2}2+1}+\frac{q^{h_2-\frac{h_1}{2}+2+m\{ \frac{h_1}m \}}}{\Phi(b)h_1h_2}\right),
\end{align*}
where $\Phi$ denotes the Euler totient function in $\F_q[T]$ and $\{\alpha\} = \alpha - \lfloor \alpha \rfloor$ denotes the fractional part of a real number $\alpha$.
\end{lem}
Note that if $\lfloor \frac{h_1}{m}\rfloor=\lfloor \frac{h_2}{m}\rfloor$, then the constraint that $g_2 \equiv g_1 \pmod{b^{\lfloor \frac{h_1}m \rfloor + 1}}$ implies that $g_2 = g_1$. Then the sum on the left-hand side above is zero unless $h_1=h_2=:h$; in that case, it is precisely the number of irreducible (not necessarily monic) polynomials of degree $h$.

We will use Lemma \ref{lem:correlations-using-random-matrix-theory-general} to prove Theorem \ref{thm:varff}, estimating the variance as $q \to \infty$ and as $m \to \infty$. To begin with, we prove \eqref{eq: variance polynomials q large} as $q \to \infty$.
\begin{proof}[Proof of \eqref{eq: variance polynomials q large}]
We start by considering the first term in \eqref{eq:variance-definition-gen}. Swapping the order of summation and separating the sums over $g_1$ and $g_2$ by degree, we obtain
\begin{align*}
& \frac 1{q^{m\ell}}
\sum_{\substack{f\in \F_q[T]\\ \deg(f)<m\ell}}\sum_{\substack{g_1,g_2\in \F_q[T]\\ \deg(g_i)< m\ell \\ g_i \equiv f \pmod{b^{\lfloor \frac{\deg(g_i)}{m}\rfloor+1}}}}\1_\mathcal{P}(g_1)\1_\mathcal{P}(g_2)\\
&= \langle \T_{q,b}\rangle_{b^\ell}+\frac 2{q^{m\ell}} \sum_{1\leq h_1<h_2< m\ell}\sum_{\substack{g_1,g_2\in \F_q[T]\\ \deg(g_i)=h_i}}\1_\mathcal{P}(g_1)\1_\mathcal{P}(g_2)
\sum_{\substack{f\in \F_q[T]\\ \deg(f)<m\ell \\ f \equiv g_i \pmod{b^{\lfloor \frac{h_i}{m}\rfloor+1}} }} 1,
\end{align*}
where the first term corresponds to the case where $h_1=h_2$ and the second term to $h_1\not =h_2$. Note that the congruence conditions on $f$ can be written as one condition on $f$ in terms of $g_2$ in addition to a condition between $g_1$ and $g_2$.
After simplifying slightly, we get that the first term in \eqref{eq:variance-definition-gen} is given by
\begin{align*}
&\langle \T_{q,b}\rangle_{b^\ell}+2\sum_{1\leq h_1<h_2< m\ell}q^{-m(\lfloor \frac{h_2}{m}\rfloor+1)}\sum_{\substack{g_1,g_2\in \F_q[T]\\ \deg(g_i)=h_i\\ g_2\equiv g_1 \pmod{b^{\lfloor \frac{h_1}{m}\rfloor+1}} }}\1_\mathcal{P}(g_1)\1_\mathcal{P}(g_2).
\end{align*}
By Lemma \ref{lem:correlations-using-random-matrix-theory-general}, the above equals
\begin{align*}
&\langle \T_{q,b}\rangle_{b^\ell}+2\frac{(q-1)^2}{\Phi(b)} \sum_{k=2}^{\ell}\frac{1}{q^{mk}}\sum_{\substack{1\leq h_1<m(k-1) \\ m(k-1) \leq h_2< mk}}\frac{q^{h_2+m\{ \frac{h_1}{m}\}}}{h_1h_2}\\
&+O\Bigg(\sum_{k=2}^{\ell}\frac{1}{q^{mk}}\sum_{\substack{1\leq h_1<m(k-1) \\ m(k-1)\leq h_2< mk}}\Big(\frac{(h_1+m)^2}{h_1h_2}q^{\frac{h_1+h_2}{2}+1}+\frac{q^{h_2-\frac{h_1}{2}+2+m\{\frac{h_1}{m}\}}}{\Phi(b)h_1h_2}\Big)\Bigg).
\end{align*}
Bounding the error term by $\ell q^{-\frac{m}{2}}+ \mathbb 1_{m \ge 2} \frac{\log \ell q^{\frac{m+1}{2}}}{m^2\Phi(b)}$ and rearranging, we obtain
\begin{align}
&\langle \T_{q,b}\rangle_{b^\ell}+2\frac{(q-1)^2}{\Phi(b)}\sum_{\substack{1\leq h_1<h_2< m\ell\\ \lfloor \frac{h_1}{m}\rfloor< \lfloor \frac{h_2}{m}\rfloor}}\frac{q^{m\{\frac{h_1}{m}\}+m\{\frac{h_2}{m}\}-m}}{h_1h_2 }+O\left(\ell q^{-\frac{m}{2}}+\mathbb 1_{m \ge 2}\frac{\log\ell q^{\frac{m+1}{2}}}{m^2\Phi(b)}\right).\label{eq:simplify}
\end{align}

Recalling that $q\rightarrow\infty$ and applying Theorem \ref{thm:preliminary-average-polynomials} and Lemma \ref{lem:rewriting the main term for the average in function fields} to evaluate $\langle \T_{q,b}\rangle_{b^\ell}$, we get that the first term of \eqref{eq:variance-definition-gen} is equal to
\begin{align*}
&\sum_{\substack{1\leq h <m \ell \\ h\equiv -1\pmod{m}}}\frac{1}{h}+2\frac{q^m}{\Phi(b)}\sum_{\substack{1\leq h_1< h_2< m\ell\\ h_1, h_2\equiv -1\pmod{m}}}\frac{1}{h_1h_2 }+O\left(\frac{\log \ell}{mq}+\frac{1}{q}+\ell q^{-\frac{m}{2}}+\mathbb 1_{m \ge 2}\frac{\log\ell q^{\frac{m+1}{2}}}{m^2\Phi(b)}\right)\\
&= \sum_{\substack{1\leq h <m \ell \\ h\equiv -1\pmod{m}}}\frac{1}{h}+
\frac{q^m}{\Phi(b)}\Bigg[\Bigg(\sum_{\substack{1\leq h <m \ell \\ h\equiv -1\pmod{m}}}\frac{1}{h}\Bigg)^2 -\sum_{\substack{1\leq h <m \ell \\ h\equiv -1\pmod{m}}}\frac{1}{h^2}\Bigg]\\
&\qquad +O\left(\frac{\log \ell}{mq}+\frac{1}{q}+\ell q^{-\frac{m}{2}}+\mathbb 1_{m \ge 2}\frac{\log\ell q^{\frac{m+1}{2}}}{m^2\Phi(b)}\right).
\end{align*}
Noting that $\frac{q^m}{\Phi(b)}-1 \rightarrow 0$ as $q\rightarrow \infty$ and again applying Theorem \ref{thm:preliminary-average-polynomials} and Lemma \ref{lem:rewriting the main term for the average in function fields}, we obtain the final result.
\end{proof}

We now turn to the proof of \eqref{eq: variance polynomials m large} as $m \to \infty$.
\begin{proof} [Proof of \eqref{eq: variance polynomials m large}]
The proof that the first term of \eqref{eq:variance-definition-gen} is equal to \eqref{eq:simplify} does not rely on taking the large $q$ limit, so we will make use of this formula. We will apply Theorem \ref{thm:preliminary-average-polynomials} and Lemmas \ref{lem:rewriting the main term for the average in function fields} and \ref{lem:correlations-using-random-matrix-theory-general}. Setting $h_i=k_im-r_i$
in \eqref{eq:simplify}, the first term of \eqref{eq:variance-definition-gen} is
\begin{align}
&\frac{1}{m}\sum_{\substack{k=1}}^\ell \frac{1}{k} +\frac{q}{(q-1)m^2}\sum_{\substack{k=1}}^\ell \frac{1}{k^2} +2\frac{(q-1)^2q^m}{\Phi(b)} \sum_{\substack{1\leq k_1<k_2\le \ell}} \sum_{\substack{r_1,r_2=1 \\ k_1m-r_1 \ne 0}}^{m} \frac{q^{-r_1-r_2}}
{(k_1m-r_1)(k_2m-r_2)}\label{eq:doublesumwithm-m}\\&+O\left(\frac{1}{m^3}+\log^2(m\ell)q^{1-\frac{m}{2}}+\ell q^{-\frac{m}{2}}\right).\nonumber
\end{align}
The contribution to the double sum in \eqref{eq:doublesumwithm-m} from terms with $k_1\ne 1$ is
\begin{equation*}
\frac 1{m^2} \sum_{2 \le k_1 < k_2 \le \ell} \frac 1{k_1k_2} \sum_{r_1 = 1}^m \frac{q^{-r_1}}{1-\frac{r_1}{k_1m}} \sum_{r_2 = 1}^m \frac{q^{-r_2}}{1-\frac{r_2}{k_2m}}.
\end{equation*}
Since $k_1, k_2 \ge 2$, we can take the Taylor expansion $\frac 1{1-x} = 1 + O(x)$ for $x \in [0,1/2]$ to get
\begin{align*}
&\frac 1{m^2} \sum_{2 \le k_1 < k_2 \le \ell} \frac 1{k_1k_2} \sum_{r_1 = 1}^m q^{-r_1} \left(1 + O\left(\frac{r_1}{k_1m}\right)\right)\sum_{r_2 = 1}^m q^{-r_2} \left(1 + O\left(\frac{r_2}{k_2m}\right)\right) \\
= &\frac 1{m^2} \sum_{2 \le k_1 < k_2 \le \ell} \frac 1{k_1k_2} \left(\frac{1-q^{-m}}{q-1} + O\left(\frac 1{qk_1m}\right)\right)\left(\frac{1-q^{-m}}{q-1} + O\left(\frac 1{qk_2m}\right)\right),
\end{align*}
where we have evaluated the geometric series and applied \eqref{eq:sumkqk}.

As $m \to \infty$, this is equal to
\begin{equation}\label{eq:bounddoublem}
 \frac{1}{m^2(q-1)^2} \sum_{2 \le k_1 < k_2 \le \ell} \frac 1{k_1k_2} + O\left(\frac{\log \ell}{m^3q^2} + \frac{\log^2\ell}{m^2 q^{m+2}}\right).
\end{equation}

The contribution to the double sum in \eqref{eq:doublesumwithm-m} from the $k_1 = 1$ terms is
\begin{align*}
\sum_{2 \le k_2 \le \ell} \sum_{r_1 = 1}^{m-1} \frac{q^{-r_1}}{m-r_1} \sum_{r_2 = 1}^m \frac{q^{-r_2}}{k_2 m - r_2} &= \frac 1{m^2} \sum_{2 \le k_2 \le \ell} \frac 1{k_2} \left(\sum_{r_1=1}^{ m/2} \frac{q^{-r_1}}{1-\tfrac{r_1}{m}} + O\left(\frac{m^2}{q^{\frac{m}{2}}}\right)\right)\sum_{r_2 = 1}^m \frac{q^{-r_2}}{1- \frac{r_2}{k_2m}} \\
&= \frac 1{m^2} \sum_{2 \le k_2 \le \ell}\frac{1}{k_2} \left(\sum_{r_1 = 1}^{m/2} \frac 1{q^{r_1}} + O\left(\frac 1{mq}\right) \right) \left(\frac{1-q^{-m}}{q-1} + O\left(\frac 1{qk_2m}\right)\right) \\
&= \frac 1{m^2(q-1)^2} \sum_{2 \le k_2 \le \ell} \frac 1{k_2} + O\left(\frac{\log \ell}{m^3q^2}\right),
\end{align*}
where in the first line we applied \eqref{eq:extra term one over m minus t} and in the second we applied a Taylor expansion to $\frac 1{1-\tfrac{r_1}{m}}$ when $\frac{r_1}{m} \in [0,1/2]$, a Taylor expansion to $\frac 1{1-\tfrac{r_2}{k_2m}}$,  as well as evaluated the sum over $r_2$ by applying \eqref{eq:sumkqk}.

Combining all of the above, as $m \to \infty$ the first term of \eqref{eq:variance-definition-gen} is
\begin{align*}
&\frac{1}{m}\sum_{k=1}^\ell \frac{1}{k}+\frac{q}{(q-1)m^2}\sum_{\substack{k=1}}^\ell \frac{1}{k^2}+\frac{q^m}{\Phi(b)m^2}\left(\left(\sum_{k=1}^\ell \frac{1}{k}\right)^2-\sum_{k=1}^\ell \frac{1}{k^2}\right)\\&+O\left(\frac{\log \ell}{m^3} \frac{q^{m}}{\Phi(b)}+\log^2(m\ell)q^{1-\frac{m}{2}}+\ell q^{-\frac{m}{2}}\right).
\end{align*}
Applying Theorem \ref{thm:preliminary-average-polynomials} and Lemma \ref{lem:rewriting the main term for the average in function fields}, we obtain the final result.

\end{proof}

The variance when $\ell \to \infty$ and $q$ and $m$ are fixed is difficult to estimate asymptotically, but below we provide heuristics in support of a conjectural estimation.  We apply Theorem \ref{thm:preliminary-average-polynomials} and Lemmas \ref{lem:rewriting the main term for the average in function fields} and  \ref{lem:correlations-using-random-matrix-theory-general}. As before, we will make use of
 \eqref{eq:variance-definition-gen}.
As we did for the large-$m$ limit, we set $h_i = k_im-r_i$ in \eqref{eq:simplify} and apply our estimates on the average value of $\langle \T_{q,b}\rangle_{b^\ell}$ as $\ell \to \infty$. This gives that the first term of \eqref{eq:variance-definition-gen} is
\begin{align}
&\frac{1-q^{-m}}{m}\log \ell+2\frac{(q-1)^2q^m}{\Phi(b)} \sum_{\substack{1 \le k_1 < k_2 \le \ell}}
\sum_{\substack{r_1,r_2=1 \\ k_1m-r_1 \ne 0}}^{m}\frac{q^{-r_1-r_2}}{(k_1m-r_1)(k_2m-r_2)}  + O_{m,q}(\ell).\label{eq:doublesum}
\end{align}
The error term in this expression overwhelms the main term as $\ell \to \infty$, as the main term is a power of $\log \ell$.
However, we expect this error term to exhibit further cancellation in practice, and so we discard it in order to conjecturally estimate the variance.

Similarly to the $m \to \infty$ case in \eqref{eq:bounddoublem}, the contribution to the double sum from terms with $k_1 \ne 1$ and as $\ell \rightarrow \infty$ is
\kommentar{\begin{equation*}
\frac 1{m^2} \sum_{2 \le k_1 < k_2 \le \ell} \frac 1{k_1k_2} \sum_{r_1 = 1}^m \frac{q^{-r_1}}{1-\frac{r_1}{k_1m}} \sum_{r_2 = 1}^m \frac{q^{-r_2}}{1-\frac{r_2}{k_2m}}.
\end{equation*}
Since $k_1, k_2 \ge 2$, we can take the Taylor expansion $\frac 1{1-x} = 1 + O(x)$ for $x \in [0,1/2]$ to get
\begin{align*}
&\frac 1{m^2} \sum_{2 \le k_1 < k_2 \le \ell} \frac 1{k_1k_2} \sum_{r_1 = 1}^m q^{-r_1} \left(1 + O\left(\frac{r_1}{k_1m}\right)\right)\sum_{r_2 = 1}^m q^{-r_2} \left(1 + O\left(\frac{r_2}{k_2m}\right)\right) \\
= &\frac 1{m^2} \sum_{2 \le k_1 < k_2 \le \ell} \frac 1{k_1k_2} \left(\frac{1-q^{-m}}{q-1} + O\left(\frac 1{qk_1m}\right)\right)\left(\frac{1-q^{-m}}{q-1} + O\left(\frac 1{qk_2m}\right)\right),
\end{align*}
where we have evaluated the geometric series and applied \eqref{eq:sumkqk}.
As $\ell \to \infty$, this is equal to}
\begin{equation*}
\frac{1}{m^2}\left(\frac{1-q^{-m}}{q-1}\right)^2\sum_{\substack{2\leq k_1<k_2\le \ell}} \frac{1}{k_1k_2} + O_{m,q}(\log \ell)
=\frac{1}{2m^2}\left(\frac{1-q^{-m}}{q-1}\right)^2\log^2\ell +O_{m,q}(\log \ell).
\end{equation*}

The contribution to the double sum  in \eqref{eq:doublesum} from terms with $k_1 = 1$ is
\begin{align*}
\sum_{r_2=1}^{m}\sum_{r_1=1}^{m-1}\frac{q^{-r_1-r_2}}{m-r_1} \sum_{\substack{2\leq k_2\le \ell}}
\frac{1}{k_2m-r_2}=&\frac{1}{m}\sum_{r_2=1}^{m}\sum_{j_1=1}^{m-1}\frac{q^{j_1-m-r_2}}{j_1} \left(\sum_{\substack{2\leq k_2 \le\ell}}\frac{1}{k_2}+O\left(\frac{r_2}{m}\right)\right )
=O_{m,q}(\log \ell).
\end{align*}

Combining all of the above, we arrive at Conjecture \ref{conj:varff}.
Note that the variance in this conjecture is dominated by a term of shape $\asymp \left(\frac{q^m}{\Phi(b)}-1\right) \frac{(1-q^{-m})^2\log^2 \ell}{m^2}$, in contrast to Theorem \ref{thm:varff}, where the variances in the large-$q$ and large-$m$ limits are dominated by terms of shape $\asymp \frac{\log \ell}{m}$. The dominating terms in Theorem \ref{thm:varff}  are absorbed into the error term in Conjecture \ref{conj:varff}, and vice versa. As in the integer case, the fact that the variance in the large-$\ell$ regime is of order $(\log \ell)^2$ instead of $\log \ell$ appears to be due to a clustering of prime truncations when the polynomials $f$ are not coprime to $b$. In particular, any polynomial $f$ that is not coprime to $b$ will have at most one prime truncation and thus will contribute $A^2$ to the variance, where $A$ is the average. As $m\rightarrow \infty$, the average $A$ is $\asymp \frac{1}{m}$, which approaches $0$, and therefore the $A^2$ contributions will be vanishingly small. Thus in the large $m$-limit, this effect disappears. However, if $b$ is fixed and $\ell\rightarrow \infty$, the average will not approach $0$, and this effect should contribute to a term in the variance of size $\asymp \log^2\ell$. This begs the question: does the same phenomenon arise if we restrict to polynomials $f$ that are relatively prime to $b$? The restricted average in this case is given by
\begin{align*}
 \langle \T_{q,b}\rangle_{b^\ell}^0:=&\frac{1}{q^{m(\ell-1)}\Phi(b)}\sum_{\substack{f\in \F_q[T]\\\deg(f)< m\ell\\(f,b)=1}}\sum_{\substack{g\in \F_q[T]\\ \deg(g)<m\ell\\
 g\equiv f \pmod{b^{\lfloor\frac{\deg(g)}{m}\rfloor+1 }}}} \1_\mathcal{P}(g) \\
  =&\frac{q^m}{\Phi(b)} \langle \T_{q,b} \rangle_{b^\ell} - \frac{1}{q^{m(\ell-1)}\Phi(b)} \sum_{\substack{g \in \F_q[T]\\g\mid b}}\1_\mathcal{P}(g)
  \sum_{\substack{f\in \F_q[T]\\ \deg(f)<m\ell\\f\equiv g \pmod{b}}}1\\
  =&\frac{q^m}{\Phi(b)}\langle \T_{q,b}\rangle_{b^\ell}+O\left(\frac{q^{\varepsilon m}}{\Phi(b)}\right).
 \end{align*}
Since the size of the set with the condition $(f,b)=1$ is $q^{m(\ell-1)}\Phi(b)$ as opposed to $q^{m\ell}$, we see no effect when $q\rightarrow \infty$. However, when $q$ and $m$ are fixed and $\ell
\rightarrow \infty$, the heuristic leading to Conjecture \ref{conj:varff} leads to
\begin{align*}
\mathrm{Var}_{b^\ell}(\T_{q,b})^0:=& \frac {q^m}{\Phi(b)q^{m\ell}}
\sum_{\substack{f\in \F_q[T]\\ \deg(f)<m\ell\\(f,b)=1}}
\Bigg(\sum_{\substack{g\in \F_q[T]\\ \deg(g)< m\ell\\g \equiv f \pmod{b^{\lfloor \frac{\deg(g)}{m}\rfloor+1}}}}\1_\mathcal{P}(g)-\langle \T_{q,b}\rangle^0_{b^\ell}\Bigg)^2
\ll_{b,q} \log \ell,
\end{align*}
which is indeed smaller than the variance predicted in Conjecture \ref{conj:varff}.

\subsection{Proof of Lemma \ref{lem:correlations-using-random-matrix-theory-general}}  \label{sec:proofsuperLemma1.2}
Our strategy for proving Lemma \ref{lem:correlations-using-random-matrix-theory-general} consists of comparing the sum we intend to evaluate with an analogous sum with the von Mangoldt function, which we estimate using Dirichlet characters to impose the congruence condition. We recall that for $f\in \F_q[T]$, the von Mangoldt function $\Lambda(f)$ is given by $\Lambda(f)=\deg(P)$ if $f$ is a positive power of $P$ with $P$ irreducible, and $\Lambda(f)=0$ otherwise. We start by proving the following statement.
\begin{lem} \label{lem:RMT-general}
Let $m$, $h_1$, and $h_2$ be positive integers such that $\lfloor \frac{h_1}{m}\rfloor<\lfloor \frac{h_2}{m}\rfloor$. Then as $q \rightarrow \infty$,
\begin{equation*}
\sum_{\substack{ g_1,g_2\in \mathbb{F}_q[T]\\
 \deg(g_i) = h_i\\g_2 \equiv g_1 \pmod{b^{\lfloor \frac{h_1}{m}\rfloor+1}} }} \Lambda(g_1)\Lambda(g_2)
=\frac{(q-1)^2q^{h_2+m\{\frac{h_1}{m}\}}}{\Phi(b)}+O\left((h_1+m)^2q^{\frac{h_1+h_2}{2}+1}\right).
\end{equation*}
\end{lem}
\begin{proof}
 We translate the congruence condition under the sum via Dirichlet characters to get
 \begin{align*}
&\sum_{\substack{ g_1,g_2\in \mathbb{F}_q[T]\\
 \deg(g_i) = h_i\\g_2 \equiv g_1  \pmod{b^{\lfloor \frac{h_1}{m}\rfloor+1}}}} \Lambda(g_1)\Lambda(g_2)=\frac{1}{\Phi\left(b^{\lfloor \frac{h_1}{m}\rfloor+1}\right)} \sum_{\chi  \pmod{b^{\lfloor \frac{h_1}{m}\rfloor+1}}}\sum_{\alpha,\beta\in \F_q^\times}\chi(\alpha)\overline{\chi}(\beta)\sum_{\substack{g_1\in \mathcal{M}_{h_1}\\ g_2\in \mathcal{M}_{h_2}}}\chi(g_1)\overline{\chi}(g_2) \Lambda(g_1)\Lambda(g_2)\\
  &=\frac{(q-1)^2}{\Phi\left(b\right)q^{m\lfloor \frac{h_1}{m}\rfloor}}\sum_{\substack{g_1\in \mathcal{M}_{h_1}\\ g_2\in \mathcal{M}_{h_2}}}\Lambda(g_1)\Lambda(g_2)+ \frac{1}{\Phi\left(b^{\lfloor \frac{h_1}{m}\rfloor+1}\right)}\sum_{\substack{\chi \pmod{b^{\lfloor \frac{h_1}{m}\rfloor+1}}\\\chi\not = \chi_0}} \sum_{\alpha,\beta \in \F_q^\times}\chi(\alpha)\overline{\chi}(\beta) M(h_1;\Lambda\chi)M(h_2;\Lambda\overline{\chi}),
\end{align*}
where $\mathcal{M}_d$ denotes the set of monic polynomials of $\F_q[T]$ of degree $d$, $\chi_0$ denotes the principal character, and
\begin{equation}\label{eq:definition of Md}
M(d;\Lambda \chi):=\sum_{f\in \mathcal{M}_d}\Lambda(f)\chi(f)
\end{equation}
(see \cite{KR3}). We have also used the fact that $\Phi(b^k)=\Phi(b)|b|^{k-1}=\Phi(b)q^{m(k-1)}$.

A character $\chi$ is said to be \emph{even} if it is trivial when viewed as a character on $\F_q$ and \emph{odd} otherwise. When $\chi$ is nonprincipal, \cite[Lemma 2.2]{KuperbergLalin} implies that
\[M(d;\Lambda \chi)=\begin{cases}
                     -1-q^{\frac{d}{2}}\mathrm{Tr}(\Theta_\chi^{d})& \chi\, \text{even},\\
                     -q^{\frac{d}{2}} \mathrm{Tr}(\Theta_\chi^{d}) & \chi\, \text{odd},
                    \end{cases}\]
where $\Theta_\chi$ denotes the Frobenius class associated to $\chi$, whose dimension is bounded by $c-1$, where $c$ denotes the degree of the conductor of $\chi$ (see \cite[Proposition 4.3]{Rosen}, and the discussion surrounding equation (10) in \cite{KuperbergLalin}).

Then, the Riemann Hypothesis (see \cite[Theorem 5.10]{Rosen}) gives
\begin{equation}\label{eq:RH}
M(d;\Lambda \chi)\leq  (c-1) q^\frac{d}{2}.
\end{equation}

Recall that (for example by \cite[Page 21, problem 14]{Rosen})
\begin{equation*}
\sum_{\substack{g\in \mathcal{M}_{d}}}\Lambda(g_1)=q^d.
\end{equation*}

By orthogonality of characters (\cite[Page 35, Proposition 4.2]{Rosen}), we have that
\[\sum_{\alpha,\beta \in \F_q^\times}\chi(\alpha)\overline{\chi}(\beta)=\sum_{\alpha\in \F_q^\times}\chi(\alpha)\sum_{\beta \in \F_q^\times}\overline{\chi}(\beta)=\begin{cases} (q-1)^2 & \chi \mbox{ even},\\
0 & \chi \mbox{ odd}.                                                                                               \end{cases}
\]

Therefore,  
\begin{align*}
  \sum_{\substack{ g_1,g_2\in \mathbb{F}_q[T]\\
 \deg(g_i) = h_i \\g_2 \equiv g_1  \pmod{b^{\lfloor \frac{h_1}{m}\rfloor+1}}}} \Lambda(g_1)\Lambda(g_2)
&=\frac{(q-1)^2q^{h_2+m\{ \frac{h_1}{m}\}}}{\Phi(b)}+\frac{(q-1)^2}{\Phi(b) q^{m \lfloor \frac{h_1}m\rfloor}}\sum_{\substack{\chi \pmod{b^{\lfloor \frac{h_1}{m}\rfloor+1}}\\\chi\not = \chi_0\\\chi\, \text{even}}} M(h_1;\Lambda\chi)M(h_2;\Lambda\overline{\chi}).
\end{align*}
The number of even characters modulo $b^{\lfloor \frac{h_1}{m}\rfloor+1}$ is $\Phi_{\text{even}}\left(b^{\lfloor \frac{h_1}{m}\rfloor+1}\right)=\frac{\Phi\left(b^{\lfloor \frac{h_1}{m}\rfloor+1}\right)}{q-1}$. Since $\chi$ and $\overline{\chi}$ are characters modulo $b^{\lfloor \frac{h_1}{m}\rfloor+1}$, the degrees of their conductors are bounded by $m\left(\lfloor \frac{h_1}{m}\rfloor+1\right)$.
The proof is therefore completed by observing that, by the Riemann Hypothesis \eqref{eq:RH},
  \begin{align*}\frac{(q-1)^2}{\Phi\left(b^{\lfloor \frac{h_1}{m}\rfloor+1}\right)}\sum_{\substack{\chi \pmod{b^{\lfloor \frac{h_1}{m}\rfloor+1}}\\\chi\not = \chi_0\\\chi\, \text{even}}} M(h_1;\Lambda\chi)M(h_2;\Lambda\overline{\chi})\ll (h_1+m)^2 q^{\frac{h_1+h_2}{2}+1}.
\end{align*}

\end{proof}

We now turn to the proof of Lemma \ref{lem:correlations-using-random-matrix-theory-general}, whose statement differs from that of Lemma \ref{lem:RMT-general} only in the aspect that $\Lambda$ is replaced by the indicator function of irreducible polynomials $\1_{\mathcal P}$. To go from $\Lambda$ to $\1_{\mathcal P}$, we must remove the weighting by the degree as well as the contribution from higher prime powers. That is, we have
\begin{align*}
\sum_{\substack{g_1,g_2 \in \mathbb F_q[T] \\ \deg(g_i) = h_i \\ g_2 \equiv g_1 \pmod{b^{\lfloor \frac{h_1}m \rfloor + 1}}}} \1_{\mathcal P}(g_1)\1_{\mathcal P}(g_2) &= \sum_{\substack{g_1,g_2 \in \mathbb F_q[T] \\ \deg(g_i) = h_i \\ g_2 \equiv g_1 \pmod{b^{\lfloor \frac{h_1}m \rfloor + 1}} }} \frac{\Lambda(g_1)\Lambda(g_2)}{h_1h_2} - \sum_{\substack{g_1,g_2 \in \mathbb F_q[T] \\ \deg(g_i) = h_i \\ g_2 \equiv g_1 \pmod{b^{\lfloor \frac{h_1}m \rfloor + 1}} \\ g_1 \text{ or }g_2 = f^k, k \ge 2}} \frac{\Lambda(g_1)\Lambda(g_2)}{h_1h_2}.
\end{align*}
By Lemma \ref{lem:RMT-general}, the first sum on the right is precisely $\frac{(q-1)^2q^{h_2+m\{\frac{h_1}{m}\}}}{\Phi(b)h_1h_2} + O(\tfrac{(h_1+m)^2}{h_1h_2} q^{\tfrac{h_1+h_2}2+1})$, which provides the desired main term in Lemma \ref{lem:correlations-using-random-matrix-theory-general}.
 We can bound the second sum by
\begin{align}\label{eq:contribution from proper powers of irreducible polynomials}
\sum_{\substack{g_1,g_2 \in \mathbb F_q[T] \\ \deg(g_i) = h_i \\ g_2 \equiv g_1 \pmod{b^{\lfloor \frac{h_1}m \rfloor + 1}} \\ g_1 \text{ or }g_2 = f^k, k \ge 2}} \frac{\Lambda(g_1)\Lambda(g_2)}{h_1h_2} &\ll \sum_{\substack{g_1,g_2 \in \mathbb F_q[T] \\ \deg(g_i) = h_i \\ g_2 \equiv g_1 \pmod{b^{\lfloor \frac{h_1}m \rfloor + 1}} \\ g_1= f^k, k \ge 2}} \frac{\Lambda(g_1)\Lambda(g_2)}{h_1h_2} + \sum_{\substack{g_2 \in \mathbb F_q[T] \\ \deg(g_2) = h_2 \\ g_2 = f^k, k \ge 2}} \frac{\Lambda(g_2)}{h_2},
\end{align}
where for the second sum we have used that $\frac{\Lambda(g_1)}{h_1} \le 1$ and that $g_1$ is determined by $g_2$.
For the first term above, we can expand via Dirichlet characters modulo $b^{\lfloor \frac{h_1}m \rfloor + 1}$ and use orthogonality of characters to get that
\begin{align*}
 &\sum_{\substack{g_1,g_2 \in \mathbb F_q[T] \\ \deg(g_i) = h_i \\ g_2 \equiv g_1 \pmod{b^{\lfloor \frac{h_1}m \rfloor + 1}} \\ g_1= f^k, k \ge 2}} \frac{\Lambda(g_1)\Lambda(g_2)}{h_1h_2}\ll  \sum_{\substack{g_1 \in \mathcal{M}_{h_1} \\ \alpha \in \F_q^\times \\ \alpha g_1= f^k, k \ge 2}} \frac{\Lambda(g_1)}{h_1}  \sum_{\substack{g_2 \in \mathcal{M}_{h_2} \\ \beta \in \F_q^\times \\ \beta g_2 \equiv \alpha g_1 \pmod{b^{\lfloor \frac{h_1}m \rfloor + 1}} }} \frac{\Lambda(g_2)}{h_2}\\
 &\ll \frac{(q-1)^2}{h_1h_2\Phi\Big(b^{\lfloor \frac{h_1}m \rfloor + 1}\Big)}\sum_{\substack{g_1\in \mathcal{M}_{h_1}\\g_2\in \mathcal{M}_{h_2}\\g_1= f^k, k \ge 2}}\Lambda(g_1)\Lambda(g_2)+\frac{(q-1)^2}{h_1h_2\Phi\Big(b^{\lfloor \frac{h_1}m \rfloor + 1}\Big)}\sum_{\substack{\chi \pmod{b^{\lfloor \frac{h_1}m \rfloor + 1}}\\ \chi \not =\chi_0\\\chi\, \text{even}}}|M_\square(h_1;\Lambda\chi)M(h_2;\Lambda\overline{\chi})|,
 \end{align*}
where $M(h_2;\Lambda \bar{\chi})$ is defined in \eqref{eq:definition of Md} and
\[M_\square(d;\Lambda\chi):=\sum_{\substack{g\in \mathcal{M}_{d}\\g= f^k, k \ge 2}}\chi(g)\Lambda(g)\]
is a sum over proper perfect powers of degree $d$, which is bounded in absolute value by 
\begin{align*}
 |M_\square(d;\Lambda\chi)|\ll \sum_{\substack{g\in \mathcal{M}_{d}\\g= h^{2k}, k \ge 1}}\Lambda(g)+\sum_{\substack{g\in \mathcal{M}_{d}\\g= h^{2k+1}, k \ge 1}}\Lambda(g)\ll q^{\frac{d}{2}}+q^{\frac{d}{3}}\ll q^{\frac{d}{2}}.
\end{align*}
Using this bound as well as the bound \eqref{eq:RH} on $|M(h_2,\Lambda \bar{\chi})|$, we have
\begin{align*}
 \sum_{\substack{g_1,g_2 \in \mathbb F_q[T] \\ \deg(g_i) = h_i \\ g_2 \equiv g_1 \pmod{b^{\lfloor \frac{h_1}m \rfloor + 1}} \\ g_1= f^k, k \ge 2}} \frac{\Lambda(g_1)\Lambda(g_2)}{h_1h_2}\ll& \frac{q^{h_2+\frac{h_1}{2}+2}}{h_1h_2\Phi\Big(b^{\lfloor \frac{h_1}m \rfloor + 1}\Big)}+\frac{q^{\frac{h_1+h_2}{2}+1}}{h_1h_2}
 \ll\frac{q^{h_2-\frac{h_1}{2}+2+m\{ \frac{h_1}m \}}}{h_1h_2\Phi(b)}+\frac{q^{\frac{h_1+h_2}{2}+1}}{h_1h_2}.
 \end{align*}
The second term in \eqref{eq:contribution from proper powers of irreducible polynomials} is bounded by
\begin{equation*}
\frac{1}{h_2}\sum_{\substack{g_2 \in \mathbb F_q[T] \\ \deg(g_2) = h_2 \\ g_2 = f^k, k \ge 2}} \Lambda(g_2)\ll \frac{q^{\frac{h_2}{2}+1}}{h_2},
\end{equation*}
which implies overall that
\begin{align*}
\sum_{\substack{g_1,g_2 \in \mathbb F_q[T] \\ \deg(g_i) = h_i \\ g_2 \equiv g_1 \pmod{b^{\lfloor \frac{h_1}m \rfloor + 1}} \\ g_1 \text{ or }g_2 = f^k, k \ge 2}} \frac{\Lambda(g_1)\Lambda(g_2)}{h_1h_2} &\ll\frac{q^{h_2-\frac{h_1}{2}+2+m\{ \frac{h_1}m \}}}{\Phi(b)h_1h_2}+\frac{q^{\frac{h_1+h_2}{2}+1}}{h_1h_2}.
\end{align*}
This completes the proof.

\section{A heuristic for the maximum number of irreducible truncations of a polynomial}\label{sec:maxff}

We now give a heuristic for the expected maximum value of the number of irreducible truncations in the polynomial setting, again making use of the Borel--Cantelli lemmas \ref{thm:BC}. Our argument will be similar to that in Section \ref{sec:maxnf}, so we will be brief in parts.

Fix a polynomial $b \in \F_q[T]$ of degree $m$. Let $p_{q,b,\ell}(k)$ denote the probability that a polynomial $f\in \F_q[T]$ with $\ell$ digits in base $b$ has at least $k$ irreducible truncations. Precisely, we define
\[p_{q,b,\ell}(k):=\mathbb{P}\left(f\in \F_q[T]\, : \, m(\ell-1)\leq \deg(f)<m\ell, \T_{q,b}(f)\geq k\right).\]

Just as in the integer case, we have the following expectation for each fixed $\ell$.
\begin{fact}\label{fact:ffq} Let $\ell>1$. There are finitely many values of $q$ and $m \ge 2$ such that for some $b \in \mathbb F_q[T]$ of degree $m$, no polynomial $f \in \mathbb F_q[T]$ with $\ell$ digits in base $b$ has $\ell$ irreducible truncations. If $m = 1$, there are finitely many $q$ such that for some $b \in \mathbb F_q[T]$ of degree $1$, no polynomial $f \in \mathbb F_q[T]$ with $\ell$ digits in base $b$ has $\ell-1$ irreducible truncations.
\end{fact}
The distinction between $m \ge 2$ and $m = 1$ comes from the fact that no degree-$0$ polynomial is irreducible, so the maximum number of irreducible truncations in this case is $\ell-1$.

For $f \in \mathbb F_q[T]$ with $\ell$ digits in base $b$, the probability that $f$ has strictly fewer than $\ell$ truncations is $(1-p_{q,b,\ell}(\ell))$. As in the integer case, Expected Fact \eqref{fact:ffq} follows via Theorem \ref{thm:BC} from the bound
\begin{equation*}
\sum_{q=1}^{\infty} (1-p_{q,b,\ell}(\ell))^{\#\{f\in \F_q[T], m(\ell-1)\leq \deg(f)<m\ell\}}<\infty,
\end{equation*}
which we expect to hold (as we will see below).

We now consider the setting where $q$ and $b$ are fixed and $\ell\rightarrow \infty$. Again the maximum number of irreducible truncations of an $\ell$-digit base-$b$ polynomial should have the shape of a function $K(\ell)$, where
\begin{equation*}
\sum_{n=1}^{\infty} \#\{f\in \F_q[T]\,:\, m(\ell-1)\leq \deg(f)<m\ell, \T_{q,b}(f)\geq cK(\ell)\}
\end{equation*}
converges whenever $c > 1$ but does not converge whenever $c \le 1$. 
\begin{fact}\label{fact:ffn}
For fixed $q$ and $b$, and let $K(\ell)$ as above. Then, as $\ell\rightarrow \infty$,
\[K(\ell)\asymp \frac{m\ell \log q}{\log \ell}.\]
\end{fact}

In support of Expected Facts \ref{fact:ffq} and \ref{fact:ffn}, we will heuristically estimate $p_{q,b,\ell}(k)$ much as we did in the integer case. 
As in Section \ref{sec:maxnf}, we consider the auxiliary probabilities
\[p^{(0)}_{q,b,\ell}(k):=\mathbb{P}\left(f\in \F_q[T]\, : \, m(\ell-1)\leq \deg(f)<m\ell, \T_{q,b}(f)= k\right)\]
and \[p^{(1)}_{q,b,\ell}(k):=\sum_{h=k}^{n}\binom{h}{k}p^{(0)}_n(h),\]
so that $p^{(0)}_{q,b,\ell}(k)$ is the probability that a polynomial of $\ell$ digits in base $b$ has precisely $k$ irreducible truncations, and $(1-q^{-m})q^{m\ell} p^{(1)}_{q,b,\ell}(k)$ counts the number of pairs $(f,J)$ such that $f$ has $\ell$ digits in base $b$, $J \subseteq [1,\ell]$ has size $k$, and all $k$ truncations of $f$ represented by the set $J$ are irreducible.
As in Section \ref{sec:maxnf}, we have the relation
\begin{equation}\label{eq:combinationff}p_{q,b,\ell}(k)= \sum_{h=k}^\ell (-1)^{h-k} \binom{h-1}{k-1}p^{(1)}_{q,b,\ell}(h)\end{equation}
for $k\geq 1$ and $p_\ell(0)=1$. We also have that $p_{q,b,\ell}(\ell) = p^{(0)}_{q,b,\ell}(\ell) =
p^{(1)}_{q,b,\ell}(\ell)$. We will proceed by estimating $p^{(1)}_{q,b,\ell}(k)$ before applying \eqref{eq:combinationff}.

We will use the random model, analogous to the Cram\'er model in the integer setting, that represents the indicator function of irreducible polynomials of degree $d$ by independent random variables $X_f$ which are $1$ with probability $\frac 1d$ and $0$ otherwise. As in the integer case, we will further approximate this probability for an $\ell$-digit polynomial in base $b$ by $\frac{1}{m(\ell-1)}$.

Analogously to the integer case, under the Cram\'er model we expect
\begin{align}\label{eq:Jff}
p^{(1)}_{q,b,\ell}(h) 
&= \sum_{\substack{ J \subset [1,\ell-1] \\ |J| = h}}  \left(\frac{q^m-1}{q^m}\right)^h \frac{1}{m^h}\prod_{d \in J} \frac 1d + \frac {1}{m\ell} \sum_{\substack{J \subset [1,\ell-1] \\ |J| = h-1}} q^n \left(\frac{q^m-1}{q^m}\right)^{h-1}\frac{1}{m^{h-1}} \prod_{d \in J} \frac 1d,
\end{align}
and thus as $q \to \infty$ or $m \to \infty$ we expect
\begin{align*}
\sum_{q=1}^\infty (1-p_{q,b,\ell}(\ell))^{(1-q^{-m})q^{m\ell} }\ll& \sum_{q=1}^\infty \left[\left(1-\frac1{\ell!q^m}\right)^{\ell!q^m}\right]^{\frac{(1-q^{-m})q^{m(\ell-1)}}{\ell!}}
\ll  \sum_{q=1}^\infty \exp \left(-\frac{q^{m(\ell-1)}}{\ell!} \right)
\ll  1,
\end{align*}
which supports Expected Fact \ref{fact:ffq}.

As $\ell \to \infty$, for $h$ much smaller than $\ell$, the first term in \eqref{eq:Jff} dominates, so we expect that
\begin{align*}
p_{q,b,\ell}(k)
\approx &  \frac{\omega_{q,b}^k}{(k-1)!}\frac{d^{k-1}}{dx^{k-1}}\left.\frac{\mathcal{F}_\ell(x)-1}{x}\right|_{x=-\omega_{q,b}},
\end{align*}
where we define $\omega_{q,b}=\left(\frac{q^m-1}{q^m}\right)\frac{1}{m}$ and where $\mathcal{F}_\ell$ is given by \eqref{eq:nf}. Following the heuristic in Section \ref{sec:maxnf}, we then expect
\begin{align*}
 p_{q,b,\ell}(k)\approx &\ell^{-\omega_{q,b}}\sum_{j=k}^{\infty}\frac{(\omega_{q,b} \log \ell)^j }{j!} \approx \ell^{-\omega_{q,b}} \frac{(\omega_{q,b} \log \ell)^k}{k!},
\end{align*}
as long as $k \ge \log^2 \ell$, say.

Following the discussion that led to Expected Fact \ref{fact:ffn}, we expect that the maximum in this case is $k = K(\ell)$ satisfying
\[\frac{(\omega_{q,b} \log \ell)^{K(\ell)} }{\ell^{\omega_{q,b}} K(\ell)!}
(1-q^{-m})q^{m\ell}\approx \frac{1}{\ell}.\]
Applying Stirling's approximation formula and again working along the lines of Section \ref{sec:maxnf}, we arrive at the expectation that
\[K(\ell)\asymp \frac{m\ell \log q}{\log \ell},\]
in line with Expected Fact \ref{fact:ffn}.

\bibliographystyle{amsalpha}

\bibliography{Bibliography}

\end{document}